\documentclass[a4paper]{article}

\usepackage{a4,latexsym,amssymb,amsmath,amsthm,paralist,german}
\usepackage[T1]{fontenc}
\usepackage{lmodern}

\usepackage[all]{xy}
\entrymodifiers={+!!<0pt,\fontdimen22\textfont2>}

\def\mapr#1{\stackrel{#1}{\longrightarrow}}

\def\surjd#1{\lower4pt\hbox{$\downarrow$}\kern-5.65pt\Big\downarrow\rlap {$\vcenter{\hbox{$\scriptstyle{{#1}}$}}$}}

\newcommand{\Spec}{\text{\it Spec}}
\newcommand{\Q}{{\mathbb Q}}
\newcommand{\Z}{{\mathbb Z}}
\newcommand{\F}{{\mathbb F}}
\newcommand{\Frob}{{\mathit{Frob}}}
\newcommand{\N}{{\mathbb N}}
\newcommand{\G}{{\mathbb G}}
\renewcommand{\O}{{\mathcal{O}}}

\newcommand{\Cl}{\text{\it Cl}}
\newcommand{\p}{{\mathfrak p}}
\newcommand{\q}{{\mathfrak q}}
\newcommand{\sm}{{\,\smallsetminus\,}}
\newcommand{\et}{\mathit{et}}
\newcommand{\nrel}{\textit{nr\!,el}}
\newcommand{\fl}{\text{\it fl}\,}
\newcommand{\cd}{\text{\it cd}\:}

\newcommand{\Gal}{\mathit{G}}

\newcommand{\Hom}{\text{\rm Hom}}
\font\emas = cmsy10 scaled\magstep2
\font\smallemas = cmsy10
\newcommand{\freeproductmed}{\mathop{\lower.2mm\hbox{\emas \symbol{3}}}\limits}
\newcommand{\freeproductsmall}{\mathop{\lower.2mm\hbox{\smallemas \symbol{3}}}}
\newcommand{\lang}{\longrightarrow}
\font\russ=wncyr10
\renewcommand{\min}{\text{\rm min}}
 \newcommand{\ressum}{\mathop{\hbox{${\displaystyle\bigoplus}'$}}\limits}
 \newcommand{\ressumsmall}{\mathop{\hbox{${\bigoplus}'$}}}
 \newcommand{\coker}{\text{\rm coker}}
\newcommand{\ds}{\displaystyle}
\renewcommand{\a}{\mathfrak{a}}

\newcommand{\nr}{\mathit{nr}}
\newcommand{\el}{\mathit{el}}
\renewcommand{\P}{\mathfrak{P}}
\newcommand{\im}{\mathrm{im}}
\newcommand{\M}{\mathcal{M}}
\newcommand{\Et}{\mathrm{Et}}
\newcommand{\FEt}{\mathrm{FEt}}
\newcommand{\Pic}{\mathrm{Pic}}
\newcommand{\Br}{\mathrm{Br}}
\newcommand{\Mor}{\mathrm{Mor}}
\newcommand{\liso}{\stackrel{\sim}{\lang}}
\newcommand{\rec}{\mathit{rec}}
\newcommand{\T}{\mathcal{T}}
\newcommand{\mQ}{\mathfrak{Q}}
\newcommand{\cor}{\mathit{cor}}

\font\russ=wncyr10
 1

\def\Sha{\hbox{\russ\char88}}

\def\Be{\hbox{\russ\char66}}

\newcommand{\back}{\hspace{-1em}}

\newtheoremstyle{alex}
  {}
  {}
  {\sl}
  {}
  {\bf}
  {.}
  {.5em}
  {}
\newtheoremstyle{alexdef}
  {}
  {}
  {\rm }
  {}
  {\bf}
  {.}
  {.5em}
  {}
\theoremstyle{alex}
 
\newtheorem{theorem}{Theorem}[section]
\newtheorem{corollary}[theorem]{Korollar}
\newtheorem{lemma}[theorem]{Lemma}

\newtheorem{proposition}[theorem]{Satz}

\theoremstyle{alexdef}
\newtheorem{definition}[theorem]{Definition}
\newtheorem{remark}[theorem]{Bemerkung}
\newtheorem{example}[theorem]{Beispiel}
\newtheorem{remarks}[theorem]{Bemerkungen}

\title{\bf\boldmath \"{U}ber Pro-$p$-Fundamentalgruppen mar\-kierter arithmetischer Kurven}
\author{Alexander Schmidt}
\date{16. Januar 2009}
\begin{document}
\hyphenation{Galois-er-wei-te-rung di-men-sio-na-len Ei-gen-schaft In-for-ma-tion Iso-mor-phis-men}
\maketitle
\begin{abstract}
Let $k$ be a global field, $p$ an odd prime number different from $\text{char}(k)$ and $S$, $T$ disjoint, finite sets of primes of $k$. Let $G_S^T(k)(p)=\Gal(k_S^T(p)|k)$ be the Galois group of the maximal $p$-extension of $k$ which is unramified outside $S$ and completely split at $T$.  We prove the existence of a finite set of primes $S_0$, which can be chosen disjoint from any given set $\M$ of Dirichlet density zero, such that the cohomology of $G_{S\cup S_0}^T(k)(p)$ coincides with the \'{e}tale cohomology of the associated marked arithmetic curve. In particular, \hbox{$\cd\, G_{S\cup S_0}^T(k)(p)=2$}. Furthermore, we can choose $S_0$ in such a way that $k_{S\cup S_0}^T(p)$ realizes the maximal $p$-extension $k_\p(p)$ of the local field $k_\p$ for all $\p\in S\cup S_0$, the cup-product
$
H^1(G_{S\cup S_0}^T(k)(p),\F_p) \otimes H^1(G_{S\cup S_0}^T(k)(p),\F_p) \to H^2(G_{S\cup S_0}^T(k)(p),\F_p)
$
is surjective and the decomposition groups of the primes in $S$ establish a free product inside $G_{S\cup S_0}^T(k)(p)$. This generalizes previous work of the author where similar results were shown in the case $T=\varnothing$ under the restrictive assumption  $p\nmid \# \Cl(k)$ and $\zeta_p\notin k$.
\end{abstract}

\section{Einf\"{u}hrung}

Es seien $k$ ein Zahlk\"{o}rper, $S$ eine endliche Stellenmenge von $k$ und $p$ eine Primzahl.  Mit $k_S(p)$ bezeichnen wir die maximale au{\ss}erhalb $S$ unverzweigte $p$-Erweiterung von $k$ und setzen
\[
G_S(k)(p)=\Gal(k_S(p)|k).
\]
Die Gruppe $G_S(k)(p)$  ist im Fall, dass $S$ alle Primteiler von $p$ enth\"{a}lt, gut (wenn auch lange nicht vollst\"{a}ndig) verstanden; siehe Kapitel~X von \cite{NSW}  f\"{u}r einen \"{U}berblick \"{u}ber die bekannten Ergebnisse.
Im Fall, dass $S$ nicht alle Stellen \"{u}ber $p$ enth\"{a}lt, war, abgesehen davon, dass die Gruppe $G_S(k)(p)$ endlich pr\"{a}sentierbar ist,  bis vor kurzem nur wenig bekannt.

Im Fall $k=\Q$ fand J. Labute \cite{La} im Jahr 2005 die ersten Beispiele von Paaren $(S,p)$ mit der Eigenschaft, dass $G_S(\Q)(p)$  die kohomologische Dimension $2$ hat und $p$ nicht in $S$ liegt. Dieser Fortschritt  war durch die neu entwickelte Theorie der milden Pro-$p$-Gruppen m\"{o}glich.  Der Autor konnte dann Labutes Ergebnisse auf beliebige Zahlk\"{o}rper ausdehnen \cite{circular,kpi1},
wobei der Fokus auf der $K(\pi,1)$-Eigenschaft lag. Diese besagt, dass die Kohomologie der Gruppe $G_S(k)(p)$ mit der \'{e}talen Kohomologie des Schemas $\Spec(\O_k)\sm S$ \"{u}bereinstimmt, was insbesondere kohomologische Dimension~$2$ impliziert. Allerdings mussten bei gegebenem Zahlk\"{o}rper~$k$ endlich viele Primzahlen von der Betrachtung ausgeschlossen werden, n\"{a}mlich die Teiler der Klassenzahl und solche~$p$ mit $\zeta_p\in k$.
Es hat sich nun herausgestellt, dass diese Einschr\"{a}nkung eliminiert werden kann, wenn man von vornherein eine allgemeinere Fragestellung betrachtet, n\"{a}mlich die nach der Gruppe $G_S^T(k)(p)$. Hier sind $S$ und $T$ endliche Stellenmengen,
$k_S^T(p)$ die maximale $p$-Erweiterung von $k$, die unverzweigt au{\ss}erhalb $S$ und voll zerlegt bei~$T$ ist, und $G_S^T(k)(p)=\Gal(k_S^T(p)|k)$. Diese Gruppe ist auch im Funktionenk\"{o}rperfall interessant.

\medskip
In dieser Arbeit zeigen wir ohne Annahmen an $S$ und $T$, dass durch Hinzunahme endlich vieler Stellen zu~$S$ eine Situation geschaffen werden kann, in der $G_S^T(k)(p)$ die kohomologische Dimension~$2$  und weitere gute Eigenschaften hat. Bei der Wahl der hinzuzunehmenden Stellen kann man \"{u}berdies eine gegebene Stellenmenge der Dirichletdichte~$0$ (also insbesondere die Stellen \"{u}ber $p$) vermeiden. Das genaue Resultat lautet folgenderma{\ss}en.

\begin{theorem} \label{haupt} Es sei $k$ ein globaler K\"{o}rper und $p$ eine ungerade von $\mathrm{char}(k)$ verschiedene Primzahl. Es seien $S$, $T$ und $\M$ paarweise disjunkte Stellenmengen von~$k$, wobei $S$ und $T$ endlich seien und $\M$ die Dirichletdichte $\delta(\M)=0$ habe. Dann existiert eine endliche zu $S\cup T\cup \M$ disjunkte Stellenmenge $S_0$ von~$k$,  so dass die folgenden Aussagen gelten.

\medskip
\begin{compactitem}
\item[\rm(i)] Die Gruppe $G_{S\cup S_0}^T(k)(p)$ hat die kohomologische Dimension~$2$ und das Cup-Produkt
\[
H^1(G_{S\cup S_0}^T(k)(p),\F_p) \otimes H^1(G_{S\cup S_0}^T(k)(p),\F_p) \lang H^2(G_{S\cup S_0}^T(k)(p),\F_p)
\]
ist surjektiv. \smallskip
\item[\rm (ii)] Es gilt \[ k_{S\cup S_0}^T(p)_\p=k_\p(p)\] f\"{u}r alle\/ $\p\in S\cup S_0$, d.h.\ die Erweiterung globaler K\"{o}rper $k_{S\cup S_0}^T(p)|k$  realisiert f\"{u}r jede Stelle\/ $\p\in S\cup S_0$ die maximale $p$-Erweiterung $k_\p(p)$ des lokalen K\"{o}rpers~$k_\p$.

\smallskip
\item[\rm (iii)] Die Zerlegungsgruppen der Stellen aus $S$ bilden ein freies Produkt in der Gruppe $G_{S\cup S_0}^T(k)(p)$, d.h.\ der nat\"{u}rliche Homomorphismus
    \[
    \freeproductmed_{\p\in S(k_{S_0}^{S\cup T}(p))} \Gal(k_\p(p)|k_\p) \lang \Gal(k_{S\cup S_0}^T(p)|k_{S_0}^{S\cup T}(p))
    \]
    ist ein Isomorphismus von Pro-$p$-Gruppen.

\smallskip
\item[\rm (iv)] F\"{u}r jeden diskreten  $G_{S\cup S_0}^T(p)$-$p$-Torsionsmodul $M$ sind die Kantenhomomorphismen der Hochschild-Serre-Spektralfolge f\"{u}r die universelle Pro-$p$-\"{U}ber\-lage\-rung
    \[
    H^i(G_{S\cup S_0}^T(k)(p), M) \lang H^i_\et (X\sm (S\cup S_0), T, M)
    \]
    Isomorphismen f\"{u}r alle $i\geq 0$. Hier ist $X$ das eindeutig bestimmte eindimensionale, regul\"{a}re, zusammenh\"{a}ngende und \"{u}ber $\Spec(\Z)$ eigentliche Schema mit Funktionenk\"{o}rper $k$ und $H^i_\et (X\sm (S\cup S_0), T, M)$ bezeichnet die \'{e}tale Kohomologie der in $T$ markierten arithmetischen Kurve $X\sm (S\cup S_0)$ mit Werten in der durch $M$ definierten Garbe (siehe Abschnitt~\ref{mark-sec}).
\end{compactitem}

\end{theorem}

\begin{remarks} 1. Der Grund f\"{u}r den Ausschluss der Primzahl $p=2$ in Theorem~\ref{haupt} ist nicht das \"{u}bliche Problem mit den reellen Stellen, sondern, dass derzeit noch keine gute Theorie milder Pro-$2$-Gruppen existiert.

\noindent
2. Es stellt sich die Frage, ob die Eigenschaften (i)--(iv) dann auch f\"{u}r die Gruppe $G_{S\cup S_0'}^T(k)(p)$ gelten, wobei $S_0'$ eine beliebige, $S_0$ umfassende und zu $T$ disjunkte endliche Stellenmenge ist. Wir werden dies in Abschnitt~\ref{erwsec} zeigen, falls keine der Stellen aus $S_0'\sm S_0$  in der Erweiterung $k_{S\cup S_0}^T(p)|k$ voll zerlegt ist (siehe Satz~\ref{enlarge}).

\noindent
3. Entsprechende Ergebnisse f\"{u}r die volle Gruppe $G_S^T(k)$, d.h.\ ohne den \"{U}bergang zur maximalen Pro-$p$-Faktorgruppe, scheinen derzeit au{\ss}er Reichweite zu sein.
\end{remarks}

Im Zahlk\"{o}rperfall spielen die Stellen \"{u}ber $p$ keine Sonderrolle in Theorem~\ref{haupt}, insbesondere wird nicht angenommen, dass  $S$ die Menge $S_p$ der Teiler von $p$ enth\"{a}lt.  Aber selbst im Zahl\-k\"{o}r\-per\-fall mit $S\supset S_p$ und $T=\varnothing$
liefert Theorem~\ref{haupt} neue Information: Aussage (ii) war bislang nur bekannt, wenn $k$ eine primitive $p$-te Einheitswurzel enth\"{a}lt (Satz von Kuz'min, siehe \cite{kuz} oder \cite{NSW},  10.8.4), sowie  f\"{u}r gewisse CM-K\"{o}rper (siehe \cite{muk} oder \cite{NSW}, X \S8 Exercise). Nach (iii) erreicht man durch Hinzunahme endlich vieler Stellen zu $S\supset S_p$, dass die Zerlegungsgruppen der Stellen \"{u}ber $p$ ein freies Produkt innerhalb der Gruppe $G_S(k)(p)$ bilden. Dies war bislang nur f\"{u}r Stellenmengen $S$ der Dirichletdichte~$1$ (siehe \cite{NSW}, 9.4.4), jedoch nicht f\"{u}r endliche Stellenmengen bekannt.

Im \glqq zahmen\grqq\ Fall $S\cap S_p=\varnothing$ mit $T=\varnothing$ wurden in \cite{kpi1} die Aussagen (i), (ii) und (iv) bewiesen, allerdings nur unter der Voraussetzung $p\nmid \# \Cl(k)$ und $\zeta_p\notin k$. Teilergebnisse im \glqq gemischten\grqq\ Fall $\varnothing \varsubsetneq S\cap S_p \varsubsetneq S_p$ wurden von K.~Wingberg \cite{wing}, Ch.~Maire \cite{maire} und D.~Vogel \cite{Vo} erzielt.

\smallskip
In Abschnitt~\ref{dualsec} werden wir sehen, dass Theorem~\ref{haupt} eine gro{\ss}e Klasse von Beispielen liefert, in denen  $G_S^T(k)(p)$ eine Pro-$p$-Dualit\"{a}tsgruppe ist (siehe Satz~\ref{dualmod}). Im Zahlk\"{o}rperfall mit $S\supset S_p$ und $T=\varnothing$ war dies nach Resultaten von Wingberg bekannt, sobald $k$ eine primitive $p$-te Einheitswurzel enth\"{a}lt  (siehe \cite{wi} oder \cite{NSW}, 10.9.8). Im Fall $\zeta_p\notin k$ gab es lediglich Ergebnisse f\"{u}r reell-abelsche Zahlk\"{o}rper und  gewisse CM-K\"{o}rper (siehe \cite{NSW}, 10.9.15 und den nachfolgenden Remark).

\smallskip
Wesentlich f\"{u}r den Beweis von Theorem~\ref{haupt} sind  arithmetische Du\-a\-li\-t\"{a}ts\-s\"{a}tze, Hasseprinzipien f\"{u}r die Kohomologie, die Theorie der milden Pro-$p$-Grup\-pen nach Labute und die Technik der freien Produkte von B\"{u}ndeln von Pro-$p$-Gruppen \"{u}ber einer topologischen Basis, wie sie in  Kapitel~IV von \cite{NSW} entwickelt wurde.

\smallskip
Der Autor dankt Ph.~Lebacque f\"{u}r seine Kommentare zu einer vorl\"{a}ufigen Version dieser Arbeit.

\section{Der \'{e}tale Situs einer markierten Kurve} \label{mark-sec}

Es sei $Y$ ein eindimensionales, noethersches, regul\"{a}res Schema und $T$ eine endliche Menge abgeschlossener Punkte auf $Y$. Mit $\Et(Y)$ bezeichnen wir wie \"{u}blich die Kategorie der \'{e}talen Morphismen von endlichem Typ $Y'\to Y$.

\begin{definition}
Die Kategorie $\Et(Y,T)$ ist die volle Unterkategorie von $\Et(Y)$, bestehend aus allen Objekten $f: Y'\to Y$ mit der Eigenschaft, dass f\"{u}r jeden abgeschlossenen Punkt $y'\in Y'$ mit $y=f(y')\in T$ die Restklassenk\"{o}rpererweiterung $k(y')|k(y)$ trivial ist. Der \'{e}tale Situs $(Y,T)_\et$ der in $T$ markierten Kurve $Y$ besteht aus der Kategorie $\Et(Y,T)$ mit surjektiven Familien als \"{U}berdeckungen.
\end{definition}

Offenbar gilt $(Y,\varnothing)_\et=Y_\et$. F\"{u}r $T_1\subset T_2$ haben wir einen nat\"{u}rlichen Morphismus $\iota: (Y,T_1)_\et\to (Y,T_2)_\et$ und somit f\"{u}r jede Garbe $F$ abelscher Gruppen auf $(Y,T_2)_\et$ und jedes $i\geq 0$ Homomorphismen
\[
H^i_\et(Y,T_2, F) \lang H^i_\et(Y,T_1,\iota^*F).
\]
F\"{u}r eine echte abgeschlossene (und damit endliche) Teilmenge $M\subset Y$ definiert man die lokalen Kohomologiegruppen $H^*_M(Y,T,-)$ auf die \"{u}bliche Art und Weise als die Rechtsableitungen des Funktors
\[
F \longmapsto \ker \big(\Gamma(Y,T,F) \to \Gamma (Y\sm M, T\sm M, F)\big).
\]
 In gleicher Weise wie f\"{u}r gew\"{o}hnliche \'{e}tale Kohomologie zeigt man Ausschneidung, d.h.\ es gilt
\[
H^i_M(Y,T, F)\cong \bigoplus_{x\in M} H^i_x(Y_x^h, T_x^h, F),
\]
wobei $Y_x^h$ die Henselisierung von $Y$ in $x$ ist, und $T_x^h$ das Urbild von $T$ in $Y_x^h$ (also $T_x^h=\{x\}$, falls $x\in T$ und ansonsten $T_x^h=\varnothing$). Auch die Konstruktion des Cup-Produkts ist vollst\"{a}ndig analog zum gew\"{o}hnlichen \'{e}talen Situs: f\"{u}r Garben $F_1,F_2$ auf $(Y,T)_\et$ und $i,j\geq 0$ haben wir eine Cup-Produkt-Paarung
\[
H^i_\et(Y, T, F_1) \times H^j_\et(Y, T, F_2) \stackrel{\cup}{\lang} H^{i+j}_\et(Y, T, F_1\otimes F_2),
\]
mit den \"{u}blichen Eigenschaften.

\bigskip
Ganz analog konstruiert man auch die Fundamentalgruppe.
Wir betrachten die volle Unterkategorie $\FEt(Y, T)$ der endlichen Morphismen $Y' \to Y$ in $\Et(Y, T)$. Diese Kategorie erf\"{u}llt die Axiome einer Galoiskategorie (\cite{sga1}, V,~4). Nach Wahl eines geometrischen Punktes $\bar x$ in $Y \sm T$ haben wir den Faserfunktor
\[
\FEt(Y, T) \lang (\textit{Mengen}),\  (Y'\to Y)\mapsto \Mor_{Y}(\bar x, Y'),
\]
dessen Automorphismengruppe per definitionem die \'{e}tale Fundamentalgruppe  von $(Y,T)$ ist. Wir bezeichnen sie mit $\pi_1^\et(Y,T,\bar x)$. Von nun an sei $Y$ zusammenh\"{a}ngend. Dann sind die Fundamentalgruppen zu verschiedenen Basispunkten isomorph, wobei der Isomorphismus bis auf innere Automorphismen kanonisch ist. Wir werden dann den Basispunkt meist von der Notation ausschlie{\ss}en. Die Fundamentalgruppe ist proendlich und klassifiziert  \'{e}tale \"{U}berlagerungen von $Y$, in denen jeder Punkt von $T$ voll zerlegt ist.
Die Gruppe $H^1_\et(Y, T,\F_p)$ klassifiziert zyklische \"{U}berlagerungen vom Grad~$p$ von $(Y, T)$. Daher haben wir  Isomorphismen
\[
H^1 (\pi_1^\et(Y,T)(p), \F_p) \liso H^1 (\pi_1^\et(Y,T), \F_p) \liso H^1_\et(Y, T,\F_p),
\]
wobei $\pi_1^\et(Y,T)(p)$ die maximale Pro-$p$-Faktorgruppe von $\pi_1^\et(Y,T)$ bezeichnet. Diesen Isomorphismus kann man gut an der Hochschild-Serre-Spektralfolge sehen.  Wir betrachten die universelle Pro-$p$-\"{U}berlagerung $\widetilde{(Y,T)}(p)$ von $(Y,T)$. Diese ist ein Pro-Objekt in $\FEt(Y, T)$ und die Projektion
\[
\widetilde{(Y,T)}(p) \lang (Y, T)
\]
ist Galoissch mit Gruppe $\pi_1^\et(Y,T)(p)$.
Ist nun $M$ ein diskreter $\pi_1^\et(Y,T)(p)$-$p$-Torsionsmodul, so erhalten wir die Spektralfolge
\[
E_2^{ij}= H^i\Big(\pi_1^\et(Y,T)(p), H^j_\et \big(\widetilde{(Y,T)}(p), M\big) \Big)  \Rightarrow H^{i+j}_\et(Y, T, M),
\]
und insbesondere Kantenhomomorphismen
\[
\phi_{i,M}: H^i(\pi_1^\et(Y,T)(p), M) \longrightarrow H^i_\et(Y, T, M), \quad i\geq 0.
\]
\begin{lemma}\label{kpi1lem} Es sei $Y$ ein eindimensionales, noethersches, regul\"{a}res Schema und $T$ eine endliche Menge abgeschlossener Punkte auf $Y$. Dann sind f\"{u}r jeden diskreten $\pi_1^\et(Y,T)(p)$-$p$-Torsionsmodul $M$ die Homomorphismen $\phi_{0,M}$ und $\phi_{1,M}$ Isomorphismen, und $\phi_{2,M}$ ist injektiv. Die folgenden Aussagen sind \"{a}quivalent.

\medskip
\begin{compactitem}
\item[\rm (i)] $\phi_{i,M}$ ist ein Isomorphismus f\"{u}r alle $i$ und jedes $M$. \smallskip
\item[\rm (ii)] $\phi_{i,\F_p}$ ist ein Isomorphismus f\"{u}r alle $i$. \smallskip
\item[\rm (iii)] $H^i_\et(\widetilde{(Y,T)}(p), \F_p)=0$ f\"{u}r alle $i\geq 1$.
\end{compactitem}
\end{lemma}

\begin{proof}
Nach Konstruktion gilt $H^1_\et(\widetilde{(Y,T)}(p), \F_p)=0$, was die erste Aussage zeigt. Gilt (ii), so folgt (i) zun\"{a}chst f\"{u}r jeden endlichen Modul $M$, weil $\pi_1^\et(Y,T)(p)$ eine Pro-$p$-Gruppe und also $\F_p$ der einzige einfache diskrete $\pi_1^\et(Y,T)(p)$-$p$-Torsionsmodul ist. Das Ergebnis \"{u}bertr\"{a}gt sich auf beliebiges $M$, weil sowohl die Kohomologie proendlicher Gruppen, als auch \'{e}tale Kohomologie mit gefilterten direkten Limiten kommutiert. Die Implikation  (iii)$\Rightarrow$(ii) kann direkt an der Spektralfolge abgelesen werden. Schlie{\ss}lich gelte~(i). Jede Klasse in $H^i_\et(\widetilde{(Y,T)}(p), \F_p)$ liegt bereits in $H^i_\et(Y',T', \F_p)$ f\"{u}r eine endliche Zwischen\"{u}berlagerung $(Y',T')$ von $\widetilde{(Y,T)}(p)\to(Y,T)$. Diese entspricht einer offenen Untergruppe $U\subset \pi_1^\et(Y,T)(p)$. Wendet man (i) auf den $\pi_1^\et(Y,T)(p)$-Modul $\mathrm{Ind}^U_{\pi_1^\et(Y,T)(p)} \F_p$ an, so sieht man, dass jedes $\alpha \in H^i_\et(Y', T', \F_p)$ bereits in $H^i_\et(Y'', T'', \F_p)$ verschwindet, wobei $(Y'',T'')$ eine geeignete endliche Zwischen\"{u}berlagerung von $\widetilde{(Y,T)}(p)\to(Y',T')$ ist. Dies zeigt (iii) und beendet den Beweis.
\end{proof}

\begin{definition} \label{kpi1def} Wenn die \"{a}quivalenten Aussagen von Lemma~\ref{kpi1lem} erf\"{u}llt sind, so sagen wir, dass $(Y, T)$ die {\bf\boldmath $K(\pi,1)$-Eigenschaft f\"{u}r $p$} hat.
\end{definition}

\begin{remark}
Lemma~\ref{kpi1lem} und Definition~\ref{kpi1def} dehnen sich in nat\"{u}rlicher Weise auf Pro-Objekte aus. Eine markierte Kurve $(Y,T)$ hat genau dann die $K(\pi,1)$-Eigenschaft f\"{u}r $p$, wenn dies f\"{u}r ihre universelle Pro-$p$-\"{U}berlagerung $\widetilde{(Y,T)}(p)$ der Fall ist.
\end{remark}

\section{Berechnung von Kohomologiegruppen}

Im Folgenden wollen wir die \'{e}tale Kohomologie markierter arithmetischer Kurven berechnen. Es sei $k$ ein lokaler oder globaler K\"{o}rper und $p\neq\text{char}(k)$ eine fixierte Primzahl. Mit $\mu_p$ bezeichnen wir die Gruppe der $p$-ten Einheitswurzeln und setzen $\delta=1$,  falls $\mu_p\subset k$ ist, und ansonsten $\delta=0$. Alle Kohomologiegruppen nehmen Werte in der konstanten Garbe $\F_p$ an, die wir von der Bezeichnung ausschlie{\ss}en.  Desweiteren benutzen wir die Bezeichnung
\[
h^i(-)=\dim_{\F_p} H^i_\et(-)\quad(=\dim_{\F_p} H^i_\et(-,\F_p)\;)\,.
\]
Wir beginnen mit einer lokalen Berechnung. F\"{u}r einen lokalen K\"{o}rper $k$, der keine Erweiterung von $\Q_p$ ist, benutzen wir die Konvention $[k:\Q_p]=0$. Desweiteren  bezeichnen  wir mit $H^i_\nr(k)$ die unverzweigte Kohomologie und setzen
\[
H^i_{/\nr}(k)= H^i(k)/H^i_\nr(k).
\]
Mit $A^\vee$ bezeichnen wir das Pontrjagin-Dual von $A$.
\renewcommand{\arraystretch}{1.2}
\begin{proposition}\label{localcoh} Es sei $k$ ein nichtarchimedischer lokaler K\"{o}rper mit von $p$ verschiedener Charakteristik.
Es sei $X=\Spec(\O_k)$, $x$ der abgeschlossene Punkt von~$X$ und $T=\varnothing$ oder $T=\{x\}$.
Dann verschwinden die lokalen \'{e}talen Kohomologiegruppen $H^i_x(X,T)$ f\"{u}r $i\leq 1$ und $i\geq 4$, und es gilt
\[
H^2_x(X,T)\cong \left\{
\begin{array} {ll}
H^1_{/\nr}(k)& \text{ falls } T=\varnothing,\\
H^1(k) & \text{ falls } T=\{x\},
\end{array}\right.
\]
also
\[
h^2_x(X,T)= \delta+[k:\Q_p] + \# T \,.
\]
Desweiteren gilt $H^3_x(X,T)\cong H^2(k)\cong \mu_p(k)^\vee$, also $h^3_x(X,T)=\delta$, und wir haben die Euler-Poincar\'{e}-Charakteristik-Formel
\[
\sum_{i=0}^3 (-1)^i h^i_x(X,T)= [k:\Q_p] + \# T\,.
\]
\end{proposition}
\renewcommand{\arraystretch}{1}

\begin{proof}
Da $X$ henselsch ist, haben wir Isomorphismen $H^i_\et(X) \cong H^i_\et(x)$ f\"{u}r alle~$i$. Da es im Fall $T=\{x\}$ keine nichttrivialen \'{e}talen \"{U}berdeckungen von $(X,T)$ gibt, erhalten wir
\[
h^i(X,T)= \left\{
\begin{array}{cl}
1 & \hbox{ f\"{u}r } i=0,\\
1-\# T & \hbox{ f\"{u}r } i=1,\\
0 & \hbox{ f\"{u}r } i\geq 2.
\end{array}
\right.
\]
Au{\ss}erdem gilt $X\sm \{x\}=\Spec(k)$, also $H^i_\et(X\sm \{x\})\cong H^i(k)$. Der lokale Dualit\"{a}tssatz (siehe \cite{NSW}, Theorem 7.2.6) zeigt $H_\et^2(X\sm \{ x\})\cong \mu_p(k)^\vee$, und nach \cite{NSW}, Corollary 7.3.9, erhalten wir
\[
h^1(X\sm \{ x\})= 1+\delta+[k:\Q_p].
\]
Schlie{\ss}lich ist der nat\"{u}rliche Homomorphismus $H^1_\et(X) \to H^1_\et(X\sm \{ x\})$  injektiv. Daher erhalten wir die Aussage des Satzes aus der exakten Ausschneidungsfolge
\[
\cdots \to H^i_x(X,T) \to H^i_\et(X,T) \to H^i_\et(X\sm \{x\}) \to H^{i+1}_x(X,T) \to \cdots\;.
\]
\end{proof}

Nun sei $k$ ein globaler K\"{o}rper und $X$ das eindeutig bestimmte eindimensionale, regul\"{a}re, zusammenh\"{a}ngende und \"{u}ber $\Spec(\Z)$ eigentliche Schema mit Funktionenk\"{o}rper $k$ (also $X=\Spec(\O_k)$, wenn $k$ ein Zahlk\"{o}rper ist, und im Funktionenk\"{o}rperfall ist $X$ eine glatte, projektive Kurve \"{u}ber einem endlichen K\"{o}rper).  Es seien $S$ und $T$ disjunkte endliche Mengen nichtarchimedischer Stellen von $k$, also disjunkte endliche Mengen abgeschlossener Punkte von $X$.

\medskip\noindent
Wir bezeichnen mit $S_\infty$ die Menge der archimedischen Stellen von $k$ ($S_\infty=\varnothing$ im Funktionenk\"{o}rperfall).   In Galois-Terminologie (und ohne Erw\"{a}hnung des Basispunktes) gilt
\[
\pi_1^\et(X\sm S, T)= G_{S\cup S_\infty}^T(k):=\Gal (k_{S\cup S_\infty}^T|k),
\]
wobei $k_{S\cup S_\infty}^T$ die maximale Erweiterung von $k$ bezeichnet, die unverzweigt au{\ss}erhalb $S\cup S_\infty$ und voll zerlegt bei $T$ ist.
F\"{u}r eine Zwischenerweiterung $K|k$ von $k_{S\cup S_\infty}^T|k$ bezeichnen wir mit
\[
(X\sm S, T)_K
\]
die Normalisierung $(X\sm S)_K$ der Kurve $X\sm S$ in $K$, die in der Stellenmenge $T(K)$ der Fortsetzungen von Stellen von $T$ auf $K$ markiert ist. Ist $K|k$ endlich, so ist $(X\sm S, T)_K$ ein Objekt in $\FEt(X\sm S,T)$, ansonsten ein Pro-Objekt.

\medskip
Sei nun $p\neq \text{char}(k)$ eine fixierte Primzahl. Mit
\[
k_{S\cup S_\infty}^{T,\el}
\]
bezeichnen wir die maximale elementar-abelsche $p$-Erweiterung von $k$ in $k_{S\cup S_\infty}^T$ und bemerken, dass
\[
G(k_{S\cup S_\infty}^{T,\el}|k) \cong H^1_\et(X\sm S, T)^\vee
\]
gilt. Mit $S_p$ bezeichnen wir die Menge der Teiler von $p$ (also $S_p=\varnothing$ im Funktionenk\"{o}rperfall).  Nimmt man im Zahlk\"{o}rperfall $p\neq 2$ oder $k$ total imagin\"{a}r an, so k\"{o}nnen wir die archimedischen Stellen ignorieren, d.h.\ es gilt
\[
G_{S}^T(k)(p)=G_{S\cup S_\infty}^T(k)(p)\quad \big(= \pi_1^\et(X\sm S, T)(p)\big).
\]
Teil (iv) unseres Hauptresultats Theorem~\ref{haupt} besagt, dass wir f\"{u}r $p\neq 2$, $p\neq \text{char}(k)$, durch Hinzunahme endlich vieler Stellen zu $S$ erreichen, dass $(X\sm S,T)$ die $K(\pi,1)$-Eigenschaft f\"{u}r $p$ hat. Dies ist im Spezialfall $S\supset S_p$ und $T=\varnothing$ schon ohne die Hinzunahme von Stellen wohlbekannt:

\begin{proposition} \label{pinS}
Es sei $k$ ein globaler K\"{o}rper, $p\neq \text{char}(k)$ eine Primzahl und $S\supset S_p$ eine endliche, nichtleere Stellenmenge von~$k$. Dann sind f\"{u}r jeden diskreten $G_{S\cup S_\infty}(k)$-$p$-Torsionsmodul $M$ die nat\"{u}rlichen Abbildungen
\[
H^i(G_{S\cup S_\infty}(k), M) \lang H^i_\et(X\sm S, M)
\]
Isomorphismen f\"{u}r alle $i\geq 0$. Ist $M$ \"{u}berdies ein diskreter $G_{S\cup S_\infty}(k)(p)$-$p$-Torsions\-modul, so sind auch die nat\"{u}rlichen Abbildungen
\[
H^i(G_{S\cup S_\infty}(k)(p), M) \lang H^i(G_{S\cup S_\infty}(k), M)
\]
Isomorphismen f\"{u}r alle $i\geq 0$.
\end{proposition}

\begin{proof}
F\"{u}r einen Beweis der ersten Aussage siehe \cite{Mi}, II, Proposition 2.9.
Die zweite Aussage ist ein Resultat von O.~Neumann, siehe \cite{NSW}, Corollary 10.4.8.
\end{proof}

\medskip
Wir benutzen nachfolgend die folgenden Bezeichnungen, wobei $p$ stets eine Primzahl ungleich $\mathrm{char}(k)$ sei.
\begin{tabbing}
\quad \= $r_1$\hspace*{1cm} \= die Anzahl der reellen Stellen von $k$  \\
\>$r_2$\ \> die Anzahl der komplexen Stellen von $k$\\
\>$r$\ \> $=r_1+r_2$, die Anzahl der archimedischen Stellen von $k$\\
\> $\delta$ \> gleich $1$, wenn $\mu_p\subset k$, und ansonsten gleich $0$\\
\> $\delta_\p$ \> gleich $1$, wenn $\mu_p\subset k_\p$, und ansonsten gleich $0$\\
\> $\Cl_S(k)$\> $=\Pic(X\sm S)$, die $S$-Idealklassengruppe von $k$\\
\> $E_{k,S}$ \> $=H^0(X\sm S,\G_m)$, die $S$-Einheitengruppe von $k$\\
\> $_n A$\> $=\ker(A \stackrel{\cdot n}{\to} A)$, wobei $A$ eine abelsche Gruppe ist, und $n\in \N$ \\
\> $A/n$\> $=\coker(A \stackrel{\cdot n}{\to} A)$, wobei $A$ eine abelsche Gruppe ist, und $n\in \N$.\\
\end{tabbing}
Falls $k$ positive Charakteristik hat, so gilt $r_1=r_2=r=0$ und $E_{k,\varnothing}$ ist die multiplikative Gruppe des endlichen Konstantenk\"{o}rpers.  Wie zu\-vor schlie{\ss}en wir in Kohomologiegruppen die konstanten Koeffizienten $\F_p$ von der Notation aus. Die Voraussetzung $p\neq 2$ oder $k$ total imagin\"{a}r im Zahlk\"{o}rperfall wird  in der ganzen Arbeit gemacht. Daher vereinbaren wir die folgende Konvention:

\medskip
\begin{minipage}{12cm}
{\em Der Begriff Stellenmenge meint stets eine Menge nichtarchimedischer Stellen.}
\end{minipage}

\medskip\noindent

\begin{proposition}\label{globalchi} Es seien $S$ und\/ $T$ disjunkte, endliche Stellenmengen des globalen K\"{o}r\-pers $k$.  Im Zahlk\"{o}rperfall sei \"{u}berdies  $p\neq 2$ oder $k$ total imagin\"{a}r. Dann ver\-schwin\-den die Gruppen $H^i_\et(X\sm S,T)=0$ f\"{u}r\/ $i\geq 4$ und es gilt
\[
\chi(X\sm S,T):= \sum_{i=0}^3 (-1)^i h^i(X\sm S,T)= r + \# T - \sum_{\p \in S\cap S_p} [k_\p:\Q_p]\;.
\]
\end{proposition}

\begin{proof} Im Fall $T=\varnothing$ und $S\supset S_p$ ist die Aussage wohlbekannt: nach \cite{Mi}, II, Theorem~2.13\,(a), hat man $\chi(X\sm S)=-r_2$ und au{\ss}erdem gilt
\[
r - \sum_{\p\in S_p} [k_\p :\Q_p] = r_1+r_2 - [k:\Q]= -r_2,
\]
wobei im Funktionenk\"{o}rperfall $[k:\Q]=0$ gesetzt sei.
Wir betrachten die exakte Ausschneidungsfolge
\[
\cdots \to \bigoplus_{\p \in S\cup T} H^i_\p(X_\p, T_\p) \to H^i_\et(X,T) \to H^i_\et (X\sm (S\cup T)) \to  \cdots \ ,
\]
wobei $X_\p$ die Komplettierung von $X$ bei $\p$ ist (die Kohomologiegruppen der Komplettierung und der Henselisierung stimmen \"{u}berein). Man erh\"{a}lt das Ergebnis f\"{u}r $S=\varnothing$ und beliebiges $T$ durch Benutzung dieser Ausschneidungsfolge f\"{u}r $S=S_p$ und durch Anwendung von Satz~\ref{localcoh}.  Das Resultat f\"{u}r beliebiges $S$ erh\"{a}lt man aus dem Fall $S=\varnothing$ mit Hilfe der Ausschneidungsfolge
\[
\cdots \to \bigoplus_{\p \in S} H^i_\p(X_\p) \to H^i_\et(X,T) \to H^i_\et (X\sm S,  T) \to  \cdots \ ,
\]
und einer erneuten Anwendung von Satz~\ref{localcoh}.
\end{proof}

\begin{corollary}\label{ausschneid2} Es seien $S$ und\/ $T$ disjunkte, endliche Stellenmengen des globalen K\"{o}rpers $k$. Dann erhalten wir eine  exakte Folge
\[
 H^1_\et(X\sm S,T) \hookrightarrow H^1_\et(X\sm S) \to \bigoplus_{\p\in T} H^1_{\nr}(k_\p) \to H^2_\et(X\sm S,T) \twoheadrightarrow H^2_\et(X\sm S)
\]
und Isomorphismen $H^i(X\sm S,T)\liso H^i(X\sm S)$ f\"{u}r $i\geq 3$.
\end{corollary}

\begin{proof}
Dies folgt aus Satz~\ref{localcoh} durch den Vergleich der Ausschneidungsfolgen f\"{u}r  $X\sm S$ und $X\sm (S\cup T)$ sowie f\"{u}r $(X\sm S,T)$ und $X\sm (S\cup T)$.
\end{proof}

Um Formeln f\"{u}r die individuellen Kohomologiegruppen angeben zu k\"{o}nnen, f\"{u}hren wir die Kummergruppe $V_S^T(k)$ ein.
Es seien $S$ und $T$ endliche disjunkte Stellenmengen von $k$ und $p\neq \textrm{char}(k)$ eine fixierte Primzahl, die wir von der Notation ausschlie{\ss}en. Wir setzen

\smallskip\noindent
\[
V^T_{S}(k):=
\{ a \in k^\times\,|\, a \in k_v^{\times p}
\hbox{ f\"{u}r } v \in {S}\,\,\hbox{und}\,\, a \in U_v k_v^{\times p}
\hbox{ f\"{u}r } v \notin T \}/ k^{\times p}, \smallskip
\]
wobei $U_v$ die Einheitengruppe des lokalen K\"{o}rpers $k_v$ bezeichnet (Konvention:
$U_v=k_v^\times$, wenn\/~$v$ archimedisch ist). In Termen von flacher Kohomologie gilt
\[
V_S^T(k)=\ker\big(H^1_\fl(X\sm (S\cup T),\mu_p) \lang \prod_{\p\in S} H^1(k_\p,\mu_p)\big).
\]

\begin{lemma}\label{VSchange}
F\"{u}r $S=\varnothing$ gibt es eine nat\"{u}rliche exakte Folge
\[
0\longrightarrow E_{k,T} /p \longrightarrow V_\varnothing^T(k) \longrightarrow \null_p \Cl_T(k) \longrightarrow 0\,.
\]
Insbesondere gilt
$
\dim_{\F_p} V_\varnothing^T(k) =  \dim_{\F_p} \null_p \Cl_T(k) + r-1+\# T + \delta
$, es sei denn $k$ ist ein Funktionenk\"{o}rper und $T=\varnothing$, in welchem Fall $\dim_{\F_p} V_\varnothing^\varnothing(k)=\dim_{\F_p} \null_p \Cl(k)+\delta$ gilt.
F\"{u}r beliebiges $S$ und jede weitere Stelle \/ $\p\notin S\cup T$  haben wir die exakte Folge
\[
0 \longrightarrow V_{S \cup \{\p\}}^T(k)\stackrel{\phi}{\longrightarrow} V_S^T(k)  \longrightarrow U_\p k_\p^{\times p}/k_\p^{\times p}.
\]
Insbesondere ist $V_S^T(k)$ stets endlich. F\"{u}r $\p\notin S_p$ gilt
$
\dim_{\F_p} \coker(\phi)\leq \delta_\p$.
\end{lemma}

\begin{proof}
Zu $a\in V_\varnothing^T(k)$ gibt es ein eindeutig bestimmtes gebrochenen $T$-Ideal mit $(a)={\mathfrak a}^p$. Die Zuordnung $a\mapsto [\a]$ induziert einen surjektiven Homomorphismus $V_\varnothing^T(k) \to \null_p\Cl_T(k)$ mit Kern $E_{k,T}/p$. Zusammen mit dem Dirichletschen Einheitensatz zeigt dies die erste Aussage. Die zweite exakte Folge ergibt sich direkt aus den Definitionen der vorkommenden Objekte. F\"{u}r $\p\notin S_p$ gilt $\dim_{\F_p} U_\p k_\p^{\times p}/k_\p^{\times p} = \delta_\p$, was die letzte Aussage zeigt.
\end{proof}

\begin{theorem} \label{globcoh} Es seien $S$ und\/ $T$ disjunkte, endliche Stellenmengen des globalen K\"{o}rpers $k$ und $p$ eine Primzahl ungleich $\mathrm{char}(k)$. Im Zahlk\"{o}rperfall sei \"{u}berdies  $p\neq 2$ oder $k$ total imagin\"{a}r. Dann gilt $H^i_\et(X\sm S,T)=0$ f\"{u}r $i\geq 4$
und
\[
\renewcommand{\arraystretch}{1.3}
\begin{array}{lcl}
h^0(X \sm S,T)&=& 1\, ,\\
h^1(X \sm S,T) & = & 1+  \ds\sum_{\p\in S} \delta_\p
               - \delta + \dim_{\F_p}V_S^T(k) +  \ds\sum_{\p\in S\cap S_p} [k_\p:\Q_p] -r - \# T,\\
h^2(X \sm S,T)& =& \ds\sum_{\p\in S} \delta_\p
               - \delta + \dim_{\F_p}V_S^T(k)  +\theta \, ,\\
h^3(X \sm S,T)&=& \theta\, .\\
\end{array}
\renewcommand{\arraystretch}{1}
\]
Hierbei ist $\theta$ gleich $1$ falls\/ $\delta =1$ und $S=\varnothing$ und ansonsten gleich~$0$.
\end{theorem}

\begin{proof}
Die Aussage \"{u}ber $h^0$ ist trivial und das Verschwinden der Kohomologie in Dimension gr\"{o}{\ss}er gleich $4$ ist bereits in Satz~\ref{globalchi} enthalten.  Artin-Verdier-Dualit\"{a}t (siehe  \cite{Ma}, 2.4 oder \cite{Mi}, Theorem~3.1) bzw.\ \'{e}tale Poincar\'{e}-Dualit\"{a}t (\cite{Mi1}, V, Corollary~2.3) zeigen
\[
H^3_\et(X)^\vee\cong\Hom_X(\F_p, \G_m)=\mu_p(k).
\]
Desweiteren impliziert das Verschwinden von $H_\et^i(X_\p,T_\p)$ f\"{u}r $i\geq 2$ zusammen mit dem lokalen Dualit\"{a}tssatz (\cite{NSW}, Theorem 7.2.6) Isomorphismen $H^3_\p(X_\p,T_\p)^\vee\cong H^2(k_\p, \F_p)^\vee\cong \mu_p(k_\p)$ f\"{u}r jeden Punkt $\p$ von~$X$. Nach Korollar~\ref{ausschneid2} ist $H^3_\et(X,T) \to H^3_\et(X)\cong \mu_p(k)^\vee$ ein Isomorphismus.
Nun betrachten wir die Ausschneidungsfolge
\[
\bigoplus_{\p\in S} H^3_\p(X) \stackrel{\alpha}{\to} H^3_\et (X,T) \stackrel{\beta}{\to} H^3_\et(X\sm S,T) \to \bigoplus_{\p\in S} H^4_\p(X_\p).
\]
Der rechts stehende Term verschwindet nach Satz~\ref{localcoh}, also ist $\beta$  surjektiv.  Die zu $\alpha$ duale Abbildung ist die nat\"{u}rliche Abbildung
\[
 \mu_p(k) \to \bigoplus_{\p\in S} \mu_p(k_\p).
\]
Diese ist injektiv, es sei denn $\delta=1$ und $S=\varnothing$.
Wir erhalten $h^3(X \sm S,T)=1$ falls $\delta=1$ und $S = \varnothing$, und ansonsten gleich~$0$.
Mit Hilfe des Isomorphismus $H^1(G_S^T(k)) \stackrel{\sim}{\rightarrow} H^1_\et(X\sm S,T)$ erhalten wir die Aussage \"{u}ber  $h^1$  aus der Berechnung der ersten Kohomologie von $G_S^T(k)$, die in \cite{NSW}, Theorem 10.7.10, durchgef\"{u}hrt ist.  Schlie{\ss}lich folgt das Ergebnis f\"{u}r $h^2$ mit Hilfe der Euler-Poincar\'{e}-Charakteristik-Formel aus Satz~\ref{globalchi}.
\end{proof}

Das Verschwinden der Kohomologie in Dimension gr\"{o}{\ss}er $2$ zusammen mit Lemma~\ref{kpi1lem} impliziert das
\begin{corollary}
\label{kpi1crit} F\"{u}r $S\neq \varnothing$ und jede Zwischenerweiterung $K|k$ von $k_S^T(p)|k$ sind die folgenden Bedingungen \"{a}quivalent.

\smallskip
\begin{compactitem}
\item[\rm (i)] $(X\sm S,T)_K$ hat die $K(\pi,1)$-Eigenschaft f\"{u}r $p$.\smallskip
\item[\rm (ii)] Der Homomorphismus $\phi_{2,\F_p}: H^2(G_S^T(K)(p)) \to H^2_\et((X\sm S,T)_K) $ ist surjektiv und es gilt $\cd\, G_S^T(K)(p)\leq 2$.
\end{compactitem}
\end{corollary}

\pagebreak
\section{Dualit\"{a}t f\"{u}r markierte arithmetische Kurven}

Nun leiten wir einen Dualit\"{a}tssatz f\"{u}r markierte arithmetische Kurven ab. Hierzu betrachten wir f\"{u}r eine Garbe $F$ auf $(X\sm S, T)_\et$  die Schafarewitsch-Tate-Grup\-pen

\[
\Sha^i(k,S,T, F)= \ker\big(H^i_\et(X\sm S, T, F) \lang \bigoplus_{\p\in S} H^i(k_\p, F) \big).
\]
Die konstanten Koeffizienten $F=\F_p$ schlie{\ss}en wir von der Notation aus, und genauso verfahren wir mit leerem~$T$. Desweiteren setzen wir
\[
\Be_S^T(k)= V_S^T(k)^\vee.
\]

\begin{theorem}  \label{bstdual} Es sei $k$ ein globaler K\"{o}rper und $p$ eine Primzahl ungleich $\mathrm{char}(k)$. Im Zahlk\"{o}rperfall sei $p\neq 2$ oder $k$ total imagin\"{a}r. Es seien $S$ und $T$ disjunkte endliche Stellenmengen von $k$. Dann gibt es
einen nat\"{u}rlichen Isomorphismus
\[
\Sha^2(k,S,T) \liso \Be_S^T(k).
\]
Insbesondere ist $\Sha^2(k,S,T)$ endlich.
\end{theorem}

\begin{remark}
Durch \"{U}bergang zum Limes \"{u}ber alle endlichen Teilmengen verallgemeinert sich Theorem~\ref{bstdual} direkt auf den Fall einer beliebigen, also nicht notwendig endlichen Stellenmenge $S$.
\end{remark}

\begin{proof}[Beweis von Theorem~\ref{bstdual}]
Wir w\"{a}hlen eine endliche Stellenmenge~$\Sigma$  so gro{\ss}, dass $S\cup T \cup S_p\subset \Sigma$ gilt und $\Pic\big((X\sm \Sigma)_{k(\mu_p)}\big)$ $p$-torsionsfrei ist. Dann verschwinden die Gruppen $\Sha^1(k,\Sigma)$ und $\Sha^1(k, \Sigma, \mu_p)$.
Hieraus folgt das Verschwinden von $\Sha^2(k,\Sigma)\cong \Sha^1(k,\Sigma,\mu_p)^\vee$ (nach Satz~\ref{pinS} und Poitou-Tate-Dualit\"{a}t, \cite{NSW}, Theorem 8.6.7) und die Exaktheit der Folge
\[
0 \lang V_S^T(k) \lang H^1_\et(X\sm \Sigma,\mu_p) \lang \prod_{\p\in S} H^1(k_\p,\mu_p) \times \back \prod_{\p\in \Sigma\backslash (S \cup T)} \back \!\!\!  H^1_{/\nr}(k_\p,\mu_p)\;.
\]
Wir dualisieren diese exakte Folge und betrachten das kommutative Diagramm
\[
\xymatrix@C=.5cm{&\ds\prod_{\p\in S} H^1(k_\p) \times \back \prod_{\p\in \Sigma\backslash (S \cup T)} \back   \!\!\! H^1_{\nr}(k_\p)\ar[r]\ar@{^{(}->}[d]&H^1_\et(X\sm \Sigma,\mu_p)^\vee\ar@{=}[d]\ar@{->>}[r]&\Be_S^T(k)\\
H^1_\et(X\sm \Sigma)\ar@{^{(}->}[r] & \ds\prod_{\p\in \Sigma} H^1(k_\p)\ar@{->>}[r]\ar@{->>}[d]&H^1_\et(X\sm \Sigma,\mu_p)^\vee\\
&\ds\prod_{\p\in T} H^1(k_\p) \times \back \prod_{\p\in \Sigma\backslash (S \cup T)} \back \!\!\!  H^1_{/\nr}(k_\p)\;.
}
\]
Die mittlere Zeile ist ein Teil der langen exakten Folge von Poitou-Tate (siehe \cite{NSW}, Theorem 8.6.10) und daher exakt. Das Schlangenlemma liefert uns die Exaktheit von
\[
H^1_\et(X\sm \Sigma) \lang \prod_{\p\in S} H^1(k_\p) \times \back \prod_{\p\in \Sigma\backslash (S \cup T)} \back   \!\!\! H^1_{/\nr}(k_\p) \lang \Be_S^T(k) \lang 0\,. \leqno \mathrm{(I)}
\]
Unter Beachtung des kommutativen Diagramms
\[
\xymatrix{H^2_\et(X\sm S, T)\ar[d]\ar[r]&H^2_\et(X\sm \Sigma)\ar@{^{(}->}[d]\\
\ds\prod_{\p\in S} H^2(k_\p)\ar@{^{(}->}[r] & \ds\prod_{\p\in \Sigma} H^2(k_\p),}
\]
erhalten wir
\[
\Sha^2(k,S,T)=\ker\big( H^2_\et(X\sm S, T)\lang H^2_\et(X\sm \Sigma)\big).
\]
Ausschneidung liefert uns somit die exakte Folge
\[
H^1_\et(X\sm \Sigma) \lang \prod_{\p\in S} H^1(k_\p) \times \back \prod_{\p\in \Sigma\backslash (S \cup T)} \back   \!\!\! H^1_{/\nr}(k_\p) \lang \Sha^2(k,S,T) \lang 0\, , \leqno \mathrm{(II)}
\]
und ein Vergleich von (I) und (II) zeigt die Behauptung des Theorems.
\end{proof}

\begin{corollary} \label{h2lokali} Es sei $k$ ein globaler K\"{o}rper und $p$ eine Primzahl ungleich $\mathrm{char}(k)$. Im Zahlk\"{o}rperfall sei $p\neq 2$ oder $k$ total imagin\"{a}r.
Es seien $S$ und $T$  disjunkte endliche Stellenmengen von $k$ mit $V_S^T(k)=0$.
Dann ist die nat\"{u}rliche Abbildung
\[
H_\et^2(X\sm S, T) \lang \prod_{\p\in S} H^2(k_\p)
\]
injektiv und ein Isomorphismus, falls $\delta=0$. Im Fall $\delta=1$ ist f\"{u}r jedes $\p_0 \in S$ die Abbildung
\[
H_\et^2(X\sm S, T) \lang \prod_{\p\in S \backslash \{\p_0\}} H^2(k_\p)
\]
ein Isomorphismus.
\end{corollary}

\begin{proof} F\"{u}r $\p \in S$ gilt
\[
H^3_\p(X,T) \cong H^2(k_\p) \cong \mu_p(k_\p)^\vee.
\]
Nach Theorem~\ref{bstdual} gilt $\Sha^2(k,S,T)=0$. Daher liefert
Ausschneidung die exakte Folge
\[
0 \to H_\et^2(X\sm S, T) \to \prod_{\p\in S} H^2(k_\p) \stackrel{\alpha}\to H_\et^3(X,T).
\]
Die zu $\alpha$ duale Abbildung ist $\mu_p(k) \to \prod_{\p\in S} \mu_p(k_\p)$, woraus die Behauptung des Korollars folgt.
\end{proof}

Schlie{\ss}lich kann man $V_S^T(k)$ in folgendem Sinne zum Verschwinden bringen.

\begin{proposition}\label{vskill}
Es seien $T$ und $\M$ disjunkte Stellenmengen, wobei $T$ endlich sei und $\M$ die Dirichletdichte $\delta(\M)=0$ habe. Dann gibt es eine endliche und zu $T\cup \M$ disjunkte Stellenmenge $S$, bestehend aus Stellen $\p$ mit $N(\p)\equiv 1 \bmod p$, so dass
\[
V_{S}^T(k)=0.
\]
\end{proposition}

\begin{proof}  Es sei $\varOmega$ die Menge aller Stellen $\p$ von $k$ mit $N(\p)\equiv 1 \bmod p$. Dann hat die Menge $\varOmega(k(\mu_p))$ der Fortsetzungen von Stellen aus $\varOmega$ nach $k(\mu_p)$ die Dirichletdichte~$1$.  Wegen $k^\times/k^{\times p}\cong H^1(k,\mu_p)$ und nach dem Hasseprinzip \cite{NSW}, Theorem 9.1.9\,(ii), ist der Homomorphismus
\[
k^\times/k^{\times p} \lang \prod_{\p\in \varOmega\backslash (T\cup\M)} k_\p^\times/k_\p^{\times p}
\]
injektiv. Da der $\F_p$-Untervektorraum $V_\varnothing^T(k)\subset k^\times/k^{\times p}$ endlichdimensional ist, finden wir eine endliche Teilmenge $S\subset \varOmega\sm (T\cup\M)$, so dass
\[
V_{S}^T(k)=\ker\big( V_\varnothing^T(k) \lang \prod_{\p\in S} k_\p^\times/k_\p^{\times p} \big)=0.
\]
\end{proof}

\section{Lokale Komponenten}

Wie zuvor sei $k$ ein globaler K\"{o}rper und  $p$ eine von $\text{char}(k)$ verschiedene Primzahl. Im Zahlk\"{o}rperfall nehmen wir an, dass $p\neq 2$ oder $k$ total imagin\"{a}r ist. Es sei  $T$ eine endliche, im Funktionenk\"{o}rperfall nichtleere, Stellenmenge mit $\null_p\Cl_T(k)=0$. Da $\Cl_T(k)$ endlich ist, gilt dann auch $\Cl_T(k)(p)=0$ und die exakte Folge
\[
0\lang E_{k,T} \lang k^\times \mapr{(v_\q)_\q}\bigoplus_{\q\notin T} \Z \lang \Cl_{T}(k) \lang 0
\]
impliziert die Exaktheit von
\[
0\lang E_{k,T}/p \lang k^\times/k^{\times p}\lang \bigoplus_{\q\notin T} \Z/p\Z \lang 0.
\]

\begin{definition}
F\"{u}r eine Stelle $\p\notin T$ bezeichne $s_\p\in k^\times/k^{\times p}$ ein Element mit $v_\p(s_\p)\equiv 1\bmod p$ und $v_\q(s_\p) \equiv 0 \bmod p$ f\"{u}r alle $\q \notin T \cup \{\p\}$. Das Element $s_\p$ ist wohldefiniert bis auf Multiplikation mit Elementen aus $E_{k,T}/p$.
\end{definition}

Wir bezeichnen  mit $k(\sqrt[p]{E_{k,T}})$ die endliche Galoiserweiterung von $k$, die durch Adjunktion der $p$-ten Wurzeln aus allen $T$-Einheiten von $k$ entsteht. Falls $\delta=0$ (d.h.\ $\mu_p\not \subset k$), so bedeutet dies,  dass auch die $p$-ten Einheitswurzeln adjungiert werden. F\"{u}r eine Stelle $\p \notin T$ ist die Erweiterung
$k(\sqrt[p]{E_{k,T}},\sqrt[p]{s_\p})$ unabh\"{a}ngig von der Auswahl von $s_\p$.

\begin{lemma}  \label{dual}
Es sei  $T$ eine endliche, im Funktionenk\"{o}rperfall nichtleere,  Menge von Stellen mit $\null_p \Cl_T(k)=0$. Desweiteren sei $\q\notin T\cup S_p$ ein in der Erweiterung $k(\sqrt[p]{E_{k,T}})|k$ voll zerlegtes Primideal. Dann ist  $k_{\{\q\}}^{T,\el}|k$ zyklisch von der Ordnung $p$ und $\q$ verzweigt in dieser Erweiterung. Im Fall $\delta=1$ gilt \"{u}berdies
\[
k(\sqrt[p]{E_{k,T}}) k_{\{\q\}}^{T,\el} = k(\sqrt[p]{E_{k,T}},\sqrt[p]{s_\q}).
\]
Eine Stelle $\p\notin T\cup \{\q\}$ der Norm $N(\p)\equiv 1 \bmod p$ zerf\"{a}llt genau dann in $k_{\{\q\}}^{T,\el}|k$, wenn $\q$ in der Erweiterung $k(\sqrt[p]{E_{k,T}},\sqrt[p]{s_\p})|k(\sqrt[p]{E_{k,T}})$ zerf\"{a}llt.
\end{lemma}

\begin{proof}
Nach Voraussetzung gilt $N(\q)\equiv 1 \bmod p$. Wegen $\null_p \Cl_T(k)=0$ liefert Lemma~\ref{VSchange} einen Isomorphismus $E_{k,T}/p\liso V_\varnothing^T(k)$. Da $\q$ vollst\"{a}ndig in $k(\sqrt[p]{E_{k,T}})|k$ zerf\"{a}llt, ist der Homomorphismus $E_{k,T}/p \to k_{\q}^{\times}/k_{\q}^{\times p}$ die Nullabbildung und wir erhalten $V_{\{\q\}}^T(k) \liso V_\varnothing^T(k)$. Insbesondere gilt
\[
\dim_{\F_p} V_{\{\q\}}^T(k)= \dim_{\F_p} E_{k,T}/p= \# T + r - 1 + \delta,
\]
was, nach Satz~\ref{globcoh}, $h^1(X\sm \{\q\}, T)=1$  impliziert. Wegen $\Cl_T(k)(p)=0$ muss $\q$ in $k_{\{\q\}}^{T,\el}$ verzweigen. Es sei nun $\delta=1$ und $k_{\{\q\}}^{T,\el}=k(\sqrt[p]{\alpha})$ mit $\alpha\in k^\times/k^{\times p}$. Nach \"{U}bergang zu $\alpha^e$ f\"{u}r ein geeignetes $e \in \F_p^\times$ gilt
\[
v_{\q}(\alpha) \equiv 1 \bmod p,\ v_\p(\alpha)\equiv 0 \bmod p \text{ f\"{u}r } \p\neq \q.
\]
Daher liegt $\alpha/s_{\q}$ in $E_{k,T}/p$, was nach Kummertheorie die Gleichheit $k(\sqrt[p]{E_{k,T}}) k_{\{\q\}}^{T,\el} \allowbreak = k(\sqrt[p]{E_{k,T}},\sqrt[p]{s_\q})$ zeigt.

Nun sei $\p\notin T\cup \{\q\}$ ein Primideal der Norm $N(\p)\equiv 1 \bmod p$. Nach Klassenk\"{o}rpertheorie zerf\"{a}llt $\p$ genau dann in der  Erweiterung $k_{\{\q\}}^{T,\el}|k$, wenn ein $s\in k^\times/k^{\times p}$ existiert, so dass gilt:

\smallskip
\begin{compactitem}
\item $v_{\mathfrak l}(s) \equiv 0 \bmod p$ f\"{u}r ${\mathfrak l} \notin T \cup \{\p,\q\}$,
\item $v_\p(s) \equiv 1 \bmod p$,
\item $s\in k_\q^{\times p}$.
\end{compactitem}

\smallskip\noindent
Existiert ein solches $s$, so gilt $s/s_\p \in E_{k,T}/p$, also $s_\p \in k_\q^{\times p}$, d.h.\ $\q$ zerf\"{a}llt in $k(\sqrt[p]{E_{k,T}},\sqrt[p]{s_\p})|k(\sqrt[p]{E_{k,T}})$. Zerf\"{a}llt umgekehrt $\q$ in dieser Erweiterung, so erf\"{u}llt $s=s_\p$ die gew\"{u}nschte Eigenschaft und $\p$ zerf\"{a}llt in $k_{\{\q\}}^{T,\el}|k$.
\end{proof}

F\"{u}r ein Element $x\in H_\et^2(X\sm S, T)$ und $\p\in S$ bezeichnen wir das Bild von $x$ unter der nat\"{u}rlichen Abbildung $H_\et^2(X\sm S, T) \to H^2(k_\p)$ mit $x_\p$ und sprechen von der $\p$-Komponente von $x$. Ist $S'\subset S$ eine Teilmenge, so fassen wir vermittels der nat\"{u}rlichen Inklusion Elemente aus $H_\et^1(X\sm S',T)$ auch als Elemente in $H_\et^1(X\sm S, T)$ auf. F\"{u}r $\p\in S$ bezeichne $\T_\p \subset G(k_S^{T,\el}|k)$ die Tr\"{a}gheitsgruppe von~$\p$. Wir nennen ein Element $\chi\in H_\et^1 (X\sm S, T)$ unverzweigt bei $\p$, wenn $\chi(\T_\p)=0$ gilt, und ansonsten verzweigt.

\begin{proposition} \label{komponenten} Es sei  $T$ eine endliche, im Funktionenk\"{o}rperfall nichtleere, Menge von Stellen mit $\null_p \Cl_T(k)=0$. Es sei $\q\notin T\cup S_p$ ein in der Erweiterung $k(\sqrt[p]{E_{k,T}})|k$ voll zerlegtes Primideal und $\chi_\q$ ein Erzeuger der nach Lemma~\ref{dual} zyklischen Gruppe $H_\et^1(X\sm \{ \q \}, T)$.
Desweiteren sei $S$ eine zu $T\cup \{\q\}$ disjunkte, endliche Menge von Stellen $\p$ mit $N(\p)\equiv 1 \bmod p$. Ist $\chi \in H_\et^1(X\sm S,T)$ ein beliebiges Element, so gelten f\"{u}r die lokalen Komponenten von $\chi\cup \chi_\q \in H_\et^2(X\sm (S\cup \{\q\}), T)$ die folgenden Aussagen:
\[
(\chi \cup \chi_\q)_\q = \left\{
\begin{array}{cl}
0 & \text{ falls } \chi(\Frob_\q)=0,\\
\neq 0 & \text{ sonst.}
\end{array}
\right.
\]
F\"{u}r $\p\in S$ gilt
\[
(\chi \cup \chi_\q)_\p = \left\{
\begin{array}{cl}
0 & \text{falls } \chi \text{ unverzweigt bei $\p$ oder $\q$ voll zerlegt} \\
& \text{in } k(\sqrt[p]{E_{k,T}},\sqrt[p]{s_\p})|k(\sqrt[p]{E_{k,T}}) \text{ ist}, \\
\neq 0 & \text{sonst.}
\end{array}
\right.
\]
\end{proposition}

\begin{proof}
F\"{u}r eine Stelle $\p$ mit $N(\p)\equiv 1 \bmod p$ ist $H^1(k_\p)$ zweidimensional, $H^2(k_\p)$ eindimensional und die Paarung $H^1(k_\p) \times H^1(k_\p) \stackrel{\cup}{\to} H^2(k_\p)$ vollkommen. Desweiteren gilt $\chi_1\cup \chi_2=0$ f\"{u}r $\chi_1,\chi_2 \in H^1_\nr(k_\p)$.

Nun ist $\chi_\q$ verzweigt bei $\q$ und $\chi$ unverzweigt bei $\q$. Daher ist $(\chi \cup \chi_\q)_\q$ genau dann von  Null verschieden, wenn das Bild von $\chi$ in $H^1_\nr(k_\q)$ nichttrivial ist, also wenn $\chi(\Frob_\q)\neq 0$.

Nun sei $\p\in S$. Es ist $\chi_\q$ unverzweigt bei $\p$ und daher $(\chi\cup\chi_\q)_\p$ genau dann von Null verschieden, wenn das Bild von $\chi_\q$ in $H^1_\nr(k_\p)$ ungleich Null und $\chi$ bei $\p$ verzweigt ist. Die erstere Bedingung ist \"{a}quivalent zu $\chi_\q(\Frob_\p)\neq 0$, und nach Lemma~\ref{dual} dazu, dass $\q$ tr\"{a}ge in der Erweiterung $k(\sqrt[p]{E_{k,T}},\sqrt[p]{s_\p})|k(\sqrt[p]{E_{k,T}})$ ist.
\end{proof}

\section{Eine Hilfserweiterung}

\begin{theorem} \label{hilf} Es seien $T$ und  $\M$  disjunkte Stellenmengen des globalen K\"{o}rpers~$k$, wobei $T$ endlich sei und $\M$ die Dirichletdichte $\delta(\M)=0$ habe.
Es gelte $p\neq 2$ und $p\neq \hbox{char}(k)$. Dann gibt es eine endliche Stellenmenge  $T_0$, sowie eine endliche, nichtleere Stellenmenge $S$, bestehend aus Stellen $\p$ mit $N(\p)\equiv 1\bmod p$, so dass die folgenden Aussagen gelten.

\smallskip
\begin{compactitem}
\item[\rm (i)]  $S \cap (T\cup T_0 \cup \M)=\varnothing$.\smallskip
\item[\rm (ii)] $(X\sm S, T \cup T_0)$ hat die $K(\pi,1)$-Eigenschaft f\"{u}r $p$.\smallskip
\item[\rm (iii)] Jedes $\p\in S$ verzweigt in $k_{S}^{T \cup T_0}(p)$.\smallskip
\item[\rm (iv)] $V_S^{T\cup T_0}(k)=0$. \smallskip
\item[\rm (v)] Das Cup-Produkt $H^1(G_S^{T\cup T_0}(k)(p)) \otimes H^1(G_S^{T\cup T_0}(k)(p)) \to H^2(G_S^{T\cup T_0}(k)(p))$ ist surjektiv.
\end{compactitem}
\end{theorem}

Im Beweis von Theorem~\ref{hilf} werden  wir das folgende hinreichende Kriterium f\"{u}r kohomologische Dimension~$2$ benutzen.

\begin{theorem}[\cite{kpi1}, Theorem 5.5] \label{mildkrit}
Es sei $p$ eine ungerade Primzahl und $G$ eine endlich pr\"{a}sentierbare Pro-$p$-Gruppe.
Angenommen es gilt $H^2(G)\neq 0$ und es gibt eine direkte Summenzerlegung $H^1(G)\cong U \oplus V$, so dass gilt:
\begin{itemize}
\item[\rm(i)] das Cup-Produkt $V\otimes V \stackrel{\cup}{\lang} H^2(G)$ ist trivial, d.h.\ $v_1\cup v_2=0$ f\"{u}r alle $v_1,v_2\in V$,
\item[\rm (ii)] das Cup-Produkt $U\otimes V \stackrel{\cup}{\lang} H^2(G)$ ist surjektiv.
\end{itemize}
Dann  gilt $\cd\, G=2$.
\end{theorem}

\begin{remark} 1. F\"{u}r Pro-$p$-Gruppen mit einer definierenden Relation war dieses Resultat bereits lange bekannt, siehe \cite{La2}.

\noindent
2. Das Kriterium in Theorem~\ref{mildkrit} liefert mehr, als nur kohomologische Dimension~$2$: die gegebene Bedingung ist hinreichend f\"{u}r die {\em Milde} von $G$. Dies wurde in \cite{kpi1}, \S 5, aus Labutes Resultaten \cite{La} abgeleitet.  Nach \cite{La}, Theorem 1.2(c), haben milde Pro-$p$-Gruppen die kohomologische Dimension~$2$.
\end{remark}

Im Verlaufe dieses Abschnitts werden wir Theorem~\ref{hilf} beweisen. Wir beginnen mit dem (einfacheren) Fall $\delta=0$, also $k$ enth\"{a}lt keine primitive $p$-te Einheitswurzel.

\begin{lemma}\label{ewnichtda}
Es sei $\delta=0$ und $S=\{\p_1,\ldots,\p_n\}$ eine endliche Menge von Stellen $\p$ mit $N(\p)\equiv 1 \bmod p$. Wir setzen $s_i=s_{\p_i}$. Dann sind die Erweiterungen
\[
k(\mu_p,\sqrt[p]{s_1},\ldots, \sqrt[p]{s_n}), \ k(\sqrt[p]{E_{k,T}}) \text{ und } k_S^{T,\el}(\mu_p)
\]
linear disjunkt \"{u}ber $k(\mu_p)$.
\end{lemma}

\begin{proof}
Nach Konstruktion ist der durch $s_1,\ldots, s_n$ in $k^\times/k^{\times p} E_{k,T}$ aufgespannte $\F_p$-Vektorraum $n$-dimensional. Nach Kummertheorie sind somit die Erweiterungen $k(\mu_p,\sqrt[p]{s_1},\ldots, \sqrt[p]{s_n})|k(\mu_p)$ und $k(\sqrt[p]{E_{k,T}})|k(\mu_p)$ linear disjunkt. Da beide im $1$-Eigenraum bez\"{u}glich des zyklotomischen Charakters $\chi_{\text{cycl}}: G(k(\mu_p)|k)\to \F_p^\times$ liegen, $k_S^{T,\el}(\mu_p)|k(\mu_p)$ jedoch im $0$-Eigenraum, folgt das Ergebnis.
\end{proof}

\begin{proof}[Beweis von Theorem~\ref{hilf} im Fall $\delta=0$]
Wir w\"{a}hlen $T_0$ so, dass $\null_p \Cl_{T\cup T_0}(k)=0$ gilt und $T\cup T_0$ nichtleer ist, falls $k$ ein Funktionenk\"{o}rper ist. Zwecks Vereinfachung der Notation ersetzen wir im folgenden $T$ durch $T\cup T_0$. Nun w\"{a}hlen wir  eine endliche und zu $T\cup \M$ disjunkte Stellenmenge $S_0$, bestehend aus Stellen $\p$  mit $N(\p)\equiv 1 \bmod p$, so dass
\[
V_{S_0\backslash \{\p\}}^T(k)=0 \text{ f\"{u}r jedes } \p \in S_0.
\]
Dies erreicht man durch zweimalige Anwendung von Satz~\ref{vskill}.   Wir numerieren die Elemente in $S_0$, also  $S_0=\{\p_1,\ldots,\p_m\}$, und setzen $s_i=s_{\p_i}$.  Aus Satz~\ref{globcoh} folgt, dass die Tr\"{a}gheitsgruppen $\T_i$ der Stellen $\p_i$, $i=1,\ldots, m$, nichttrivial und von der Ordnung $p$ sind. Wir erweitern nun $S_0$ um weitere $m$ Primideale in der folgenden Weise:

Wir w\"{a}hlen Fortsetzungen $\P_1,\ldots,\P_m$ von $\p_1,\ldots,\p_m$ nach $k(\mu_p)$ und betrachten f\"{u}r ein Primideal $\mQ$ in $k(\mu_p)$ und  $a\in \{1,\ldots, m\}$ die folgende Bedingung $(B_a)$:

\medskip\noindent
\begin{compactitem}
\item $\mQ \notin T(k(\mu_p)) \cup \M(k(\mu_p))$.\smallskip
\item $\Frob_{\mQ} \notin \T_{\P_a} \subset G(k_{S_0}^{T,\el}(\mu_p)|k(\mu_p))$.\smallskip
\item F\"{u}r alle $b\neq a$ zerf\"{a}llt $\mQ$ in $k(\mu_p,\sqrt[p]{s_b})|k(\mu_p)$.\smallskip
\item $\mQ$ ist tr\"{a}ge in $k(\mu_p,\sqrt[p]{s_a})|k(\mu_p)$.\smallskip
\item $\mQ$ zerf\"{a}llt vollst\"{a}ndig in $k(\sqrt[p]{E_{k,T}})|k(\mu_p)$.\smallskip
\end{compactitem}

\medskip\noindent
Da das vollst\"{a}ndige Zerfallen von $\mQ$ in $k(\sqrt[p]{E_{k,T}})$ zu den Bedingungen geh\"{o}rt, ist $(B_a)$ unabh\"{a}ngig von der Auswahl  der $s_i$. Nach Lemma~\ref{ewnichtda} finden wir mit Hilfe des Tschebotarjowschen  Dichtigkeitssatzes ein  Primideal $\mQ_1$ in $k(\mu_p)$, das die Bedingung $(B_1)$ erf\"{u}llt.   Dann setzen wir $\q_1=\mQ_1 \cap k$.  Nach Lemma~\ref{dual} ist die Erweiterung $k_{\{\q_1\}}^{T,\el}$ zyklisch von der Ordnung $p$ und bei $\q_1$ verzweigt. Nun w\"{a}hlen wir nacheinander Stellen $\mQ_2,\ldots,\mQ_m$ in $k(\mu_p)$ und setzen jeweils $\q_a=\mQ_a\cap k$,  so dass gilt

\medskip
\begin{compactitem}
\item $\mQ_a$ erf\"{u}llt Bedingung $(B_a)$, und\smallskip
\item $\mQ_a$ ist f\"{u}r $b<a$  zerlegt in $k_{\{\q_b\}}^{T,\el}(\mu_p)|k(\mu_p)$ und in $k(\mu_p,\sqrt[p]{s_{\q_b}})|k(\mu_p)$.
\end{compactitem}

\medskip\noindent
Dies ist m\"{o}glich, weil nach der Wahl von $\mQ_1,\ldots,\mQ_{a-1}$ und nach Lemma~\ref{ewnichtda} die Erweiterungen
\[
k(\mu_p,\sqrt[p]{s_1},\ldots, \sqrt[p]{s_m}, \sqrt[p]{s_{\q_1}},\ldots, \sqrt[p]{s_{\q_{a-1}}} ), \ k(\sqrt[p]{E_{k,T}}) \text{ und } k_{S_0\cup \{ \q_1,\ldots, \q_{a-1}\}}^{T,\el}(\mu_p)
\]
linear disjunkt \"{u}ber $k(\mu_p)$ sind.  Wir setzen
\[
S=\{\p_1,\ldots,\p_m,\q_1,\ldots,\q_m\}.
\]
Es gilt $h^2(X\sm S,T)=2m$ und nach Korollar~\ref{h2lokali} ist die nat\"{u}rliche Abbildung
\[
H_\et^2(X\sm S, T) \lang \prod_{i=1}^m H^2(k_{\p_i}) \oplus \prod_{i=1}^m H^2(k_{\q_i})
\]
ein Isomorphismus.  Es sei $\eta_a$ ein Erzeuger von $H_\et^1(X\sm \{ \q_a\}, T)$. Wir betrachten den durch $\eta_1,\ldots,\eta_m$ in $H_\et^1(X\sm S, T)$ aufgespannten $m$-di\-men\-sio\-na\-len Vektorraum~$V$. Dann gilt
\[
H_\et^1(X\sm S, T) \cong H_\et^1(X\sm S_0, T) \oplus V.
\]
Nach Satz~\ref{komponenten} gilt f\"{u}r $a,b, i\in \{1,\ldots, m\} $
\[
(\eta_a \cup \eta_b)_{\q_i}= 0 = (\eta_a \cup \eta_b)_{\p_i},
\]
also ist das Cup-Produkt $V\otimes V \to H_\et^2(X\sm S, T)$ trivial. Wir behaupten, dass das Cup-Produkt
\[
H_\et^1(X\sm S_0, T) \otimes V \lang H_\et^2(X\sm S, T)
\]
surjektiv ist. Dazu w\"{a}hlen wir  Elemente $\chi_a,\psi_a \in H_\et^1(X\sm S_0, T)$, $a=1,\ldots, m$, so dass
\[
\chi_a(\T_{\p_a})\neq 0,\ \chi_a(\Frob_{\q_a})=0,\ \psi_a(\Frob_{\p_a})\neq 0.
\]
Wir behaupten, dass die Elemente $\chi_1\cup \eta_1,\ldots, \chi_m\cup \eta_m, \psi_1\cup \eta_1,\ldots, \psi_m\cup \eta_m$ den $2m$-dimensionalen Vektorraum $H_\et^2(X\sm S, T) $ erzeugen. Dazu betrachten wir die Matrix

\smallskip
\[
\left(
\begin{array}{cccccc}
(\chi_1 \cup \eta_1)_{\p_1}& \cdots & (\chi_1 \cup \eta_1)_{\p_m} & (\chi_1 \cup \eta_1)_{\q_1}& \cdots & (\chi_1 \cup \eta_1)_{\q_m}\\
&\vdots && \vdots&&\vdots \\
(\chi_m \cup \eta_m)_{\p_1}&\cdots &(\chi_m \cup \eta_m)_{\p_m}& (\chi_m \cup \eta_m)_{\q_1}& \cdots&(\chi_m \cup \eta_m)_{\q_m}\\
(\psi_1 \cup \eta_1)_{\p_1}& \cdots & (\psi_1 \cup \eta_1)_{\p_m} & (\psi_1 \cup \eta_1)_{\q_1}& \cdots & (\psi_1 \cup \eta_1)_{\q_m}\\
&\vdots && \vdots&&\vdots \\
(\psi_m \cup \eta_m)_{\p_1}&\cdots &(\psi_m \cup \eta_m)_{\p_m}& (\psi_m \cup \eta_m)_{\q_1}& \cdots&(\psi_m \cup \eta_m)_{\q_m}\\
\end{array} \right) \lower1.35cm\hbox{.}
\]

\smallskip\noindent
Bezeichnen wir ein von Null verschiedenes Element mit $*$ und ein beliebiges mit $?$, so hat nach Satz~\ref{komponenten} und unseren Wahlen diese Matrix die Gestalt

\smallskip
\[
\left(
\begin{array}{cccccccc}
*&0& \cdots & 0 & 0&0& \cdots & 0\\
0&*& \cdots & 0 & 0&0& \cdots & 0\\
&&\ddots && &&\vdots& \\
0&0&\cdots &*& 0& 0 &\cdots&0\\
?&?& \cdots  &? & *&0& \cdots & 0\\
?&?& \cdots  &? & 0&*& \cdots & 0\\
&&\vdots && \vdots&&\ddots \\
?&?& \cdots  &? & 0&0& \cdots & *\\
\end{array} \right) \lower1.7cm\hbox{,}
\]

\smallskip\noindent
was die Behauptung zeigt. Insbesondere ist das Cup-Produkt
\[
H_\et^1(X\sm S, T) \otimes H_\et^1(X\sm S, T) \to H_\et^2(X\sm S, T)
\]
surjektiv. Da es \"{u}ber die Inklusion
\[
H^2(G_S^T(k)(p)) \hookrightarrow H_\et^2(X\sm S, T)
\]
faktorisiert, ist diese ein Isomorphismus. Damit erf\"{u}llt die Gruppe $G_S^T(k)(p)$ die Voraussetzungen von Theorem~\ref{mildkrit} und wir erhalten $\cd\, G_S^T(k)(p)=2$. Korollar~\ref{kpi1crit} impliziert nun, dass $(X\sm S, T)$ die $K(\pi,1)$-Eigenschaft f\"{u}r $p$ hat. Nach der Wahl von $S_0$  gilt schlie{\ss}lich $V_{S\sm \{\p\}}^T(k)=0$ f\"{u}r jedes $\p\in S$, woraus folgt, dass jedes $\p\in S$ in $k_S^{T,\el}$ verzweigt, also erst recht in $k_S^T(p)$.
\end{proof}

Im Fall $\delta=1$ brauchen wir das folgende Lemma.

\begin{lemma} \label{einheitswda} Es gelte $\delta=1$ und $T$ sei eine endliche, nichtleere Stellenmenge mit $V_T^\varnothing(k)=0$ und $S_p\subset T$. Ist $k$ ein Zahlk\"{o}rper, so nehmen wir zudem $k$ als total imagin\"{a}r an (automatisch falls $p\neq 2$). Dann gilt $\null_p \Cl_T(k)=0$ und
\[
k_T^\el = k(\sqrt[p]{E_{k,T}}).
\]
Zus\"{a}tzlich sei $S=\{\p_1,\ldots,\p_n\}$ eine endliche, zu $T$ disjunkte Stellenmenge. Dann haben wir eine Inklusion
\[
 k_S^{T,\el} \subset k_T^\el(\sqrt[p]{s_1},\ldots,\sqrt[p]{s_n}),
\]
wobei $s_i=s_{\p_i}$, $i=1,\ldots,n$. Die folgenden Aussagen sind \"{a}quivalent:

\smallskip
\begin{compactitem}
\item[\rm (i)] Die Elemente $\Frob_{\p_1}, \ldots, \Frob_{\p_n} $ erzeugen $G(k_T^\el|k)$.\smallskip
\item[\rm (ii)]  $V_S^T(k)=0$.
\end{compactitem}

\smallskip\noindent
F\"{u}r eine Teilmenge $I \subset \{1,\ldots,n\}$  sind  \"{a}quivalent:

\smallskip
\begin{compactitem}
\item[\rm (a)] Die Elemente $\{\Frob_{\p_i}, i\in I\}$ sind linear unabh\"{a}ngig in $G(k_T^\el|k)$.
\item[\rm (b)] Die Erweiterungen $k_T^\el$, $k_S^{T,\el}$ und $k(\sqrt[p]{s_i}, i\in I)$ sind linear disjunkt \"{u}ber $k$.
\end{compactitem}
\end{lemma}

\begin{remark}
Sind  $\Frob_{\p_1}, \ldots, \Frob_{\p_n} $ linear unabh\"{a}ngig in  $G(k_T^\el|k)$, so scheint die Inklusion $k_S^{T,\el} \subset k_T^\el(\sqrt[p]{s_1},\ldots,\sqrt[p]{s_n})$ im Widerspruch zu Aussage (b) f\"{u}r $I=\{1,\ldots,n\}$  zu stehen. Aber in diesem Fall gilt $k_S^{T,\el}=k$.
\end{remark}

\begin{proof}[Beweis von Lemma~\ref{einheitswda}] Nach Theorem~\ref{bstdual} gilt $\Sha^2(k,T)=0$, woraus mit Poitou-Tate-Dualit\"{a}t (\cite{NSW}, Theorem 8.6.7) $\Sha^1(k,T,\mu_p)=0$ folgt. Wegen $\delta=1$ und nach \cite{NSW}, Lemma~8.6.3 erhalten wir $\Hom(\Cl_T(k),\Z/p\Z)= \Sha^1(k,T)=0$. Da $\Cl_T(k)$ endlich ist, folgt hieraus $\null_p\Cl_T(k)=0$. Nach dem Dirichletschen Einheitensatz gilt $\dim_{\F_p} E_{k,T}/p = \# T + r$.  Nach Theorem~\ref{globcoh} erhalten wir
\[
h^1(X\sm T)= \# T + \sum_{\p\in S_p} [k_\p:\Q_p] - r = \# T +r,
\]
was $k(\sqrt[p]{E_{k,T}})=k_T^\el$ impliziert.
Nach Lemma~\ref{VSchange} haben wir einen Isomorphismus
\[
E_{k,T}/p \liso V_\varnothing^T(k).
\]
Sei nun $S$ eine endliche, zu $T$ disjunkte Stellenmenge und $k(\sqrt[p]{\alpha})$, $\alpha \in k^\times/k^{\times p}$, eine zyklische Teilerweiterung von $k_S^{T,\el}|k$. Dann gilt $\alpha\in V_T^S(k)$ und wir finden Exponenten $a_1,\ldots,a_n$, so dass
$\alpha \cdot \bar s_1^{a_1}\cdots \bar s_n^{a_n} \in V_\varnothing^T(k)= E_{k,T}/p$. Daher gilt $k(\sqrt[p]{\alpha}) \subset k(\sqrt[p]{E_{k,T}},\sqrt[p]{s_1},\ldots,\sqrt[p]{s_n})$, und folglich
\[
k_S^{T,\el} \subset k_T^\el(\sqrt[p]{s_1},\ldots,\sqrt[p]{s_n}).
\]
Als n\"{a}chstes beweisen wir die \"{A}quivalenz der Bedingungen (i) und (ii). Es gilt $V_S^T(k)=0$ genau dann, wenn die nat\"{u}rliche Abbildung $E_{k,T}/p \to \prod_{i=1}^n k_{\p_i}^\times/k_{\p_i}^{\times p}$ injektiv ist. Dies ist \"{a}quivalent dazu, dass f\"{u}r kein von Eins verschiedenes Element $e\in E_{k,T}/p$  die zyklische Erweiterung $k(\sqrt[p]{e})|k$ voll zerlegt bei $S$ ist. Daher ist (ii) \"{a}quivalent dazu, dass
die Elemente $\Frob_{\p_1}, \ldots, \Frob_{\p_n}$ die Galoisgruppe $G(k(\sqrt[p]{E_{k,T}})|k)=G(k_T^\el|k)$ erzeugen.  Dies zeigt die \"{A}quivalenz zwischen (i) und~(ii).

\medskip
Wir bezeichnen mit $I_{k,T}$ die Gruppe der $T$-Idele  und mit $C_T(k)$ die Gruppe $T$-Idelklassen von $k$. Wegen $\Cl_T(k)(p)=0$ haben wir nach \cite{NSW}, Proposition 8.3.5, die exakte Folge
\[
0 \lang E_{k,T} \otimes \Z_p \lang I_{k,T} \otimes \Z_p \lang C_T(k) \otimes \Z_p \lang 0\,.
\]
Nun betrachten wir eine Teilmenge $I\subset \{1,\ldots, n\}$ und es sei $H_I \subset k^\times/k^{\times p}$ die von den $s_i$, $i\in I$, erzeugte Untergruppe. Wegen $S_p\subset T$ gilt
\[
k_S^{T,\el}=k\left(\sqrt[p]{V_T^S(k)}\right).
\]
Nach Kummertheorie sind die Erweiterungen $k_T^\el$, $k_S^{T,\el}$ und $k(\sqrt[p]{s_i}, i\in I)$ genau dann linear disjunkt \"{u}ber $k$, wenn der Homomorphismus
\[
E_{k,T}/p \times V_T^S(k) \times H_I  \lang  k^\times/k^{\times p}
\]
injektiv ist. Wegen $H_I \cap E_{k,T}/p = 1$ ist dies \"{a}quivalent zu
\[
( H_I \cdot E_{k,T}/p ) \cap V_T^S(k) =  1,
\]
und damit zur Injektivit\"{a}t des Homomorphismus
\[
H_I \cdot E_{k,T}/p \lang \prod_{\p\in T} k_\p^\times/k_\p^{\times p}=I_{k,T}/p,
\]
also zur Injektivit\"{a}t der Komposition
\[
H_I \lang (I_{k,T}/p)/ \im (E_{k,T}/p) \liso C_T(k)/p \mathop{\longrightarrow}\limits_\rec^\sim G(k_T^\el|k).
\]
Nach Klassenk\"{o}rpertheorie bildet sich $s_i$ auf $\Frob_{\p_i}^{-1}\in G(k_T^\el|k)$ ab, weshalb die Injektivit\"{a}t dieser Abbildung \"{a}quivalent zu Aussage (a) ist.
\end{proof}

\begin{proof}[Beweis von Theorem~\ref{hilf} im Fall $\delta=1$] Wir w\"{a}hlen eine endliche Stellenmenge $T_0$ so gro{\ss}, dass $V_{T\cup T_0}^\varnothing(k)=0$ und $S_p\subset T\cup T_0$ gilt. Zwecks Vereinfachung der Notation ersetzen wir $T$ durch $T\cup T_0$.
Nun w\"{a}hlen wir zu jedem von Null verschiedenen Element $g\in G(k_T^\el|k)$ ein nicht in $T\cup \M$ liegendes Primideal $\p_g$ mit $g=\Frob_{\p_g}$. Die Gesamtheit dieser Primideale nennen wir ${S_0}$ und wir w\"{a}hlen eine Numerierung
\[
{S_0}=\{\p_0, \ldots, \p_m\}.
\]
Wie vorher setzen wir $s_i=s_{\p_i}$. Nach Lemma~\ref{einheitswda} gilt  $V_{S_0}^T(k)=0$, $k_T^\el= k(\sqrt[p]{E_{k,T}})$ und $ k_{S_0}^{T,\el} \subset k_T^\el(\sqrt[p]{s_0},\ldots,\sqrt[p]{s_m})$. Nach Satz~\ref{globcoh} erhalten wir $h^2(X\sm {S_0},T)=m$ und $h^1(X\sm T)=\# T +r:=n$.

\smallskip
Nun sei $a$, $1\leq a\leq m$, ein Index. Wir w\"{a}hlen  eine Teilmenge $I_a \subset \{1,\ldots,m\}$ mit $a\notin I_a$ der Kardinalit\"{a}t $n-1$, so dass sowohl
$(\Frob_{\p_0}, \{ \Frob_{\p_i}\}_{i\in I_a})$ als auch $(\Frob_{\p_a}, \{ \Frob_{\p_i}\}_{ i\in I_a})$ eine Basis von $G(k_T^\el|k)$ ist. Dies ist m\"{o}glich: sind $\Frob_{\p_0}$ und $\Frob_{\p_a}$ linear abh\"{a}ngig in $G(k_T^\el|k)$, so erg\"{a}nzen wir $\Frob_{\p_0}$ zu einer Basis. Sind die beiden Elemente linear unabh\"{a}ngig, so w\"{a}hlen wir $a'$ mit $\Frob_{\p_{a'}}=\Frob_{\p_0} + \Frob_{\p_a}$ und erg\"{a}nzen  $(\Frob_{\p_0}, \Frob_{\p_{a'}})$ zu einer Basis.

Nach Lemma~\ref{einheitswda} sind die Erweiterungen $k_T^\el|k$ (Grad $p^n$), $k_{S_0}^{T,\el}$ (Grad $p^{m+1-n}$) und $k(\sqrt[p]{s_a}, \sqrt[p]{s_i}, i\in I_a)$ (Grad $p^n$) linear disjunkt \"{u}ber $k$. Aus Gradgr\"{u}nden ist ihr Kompositum gleich $k_T^\el(\sqrt[p]{s_0},\ldots,\sqrt[p]{s_m})$ (Grad $p^{m+1+n}$).

Nach Wahl von $I_a$ gilt die gleiche Aussage auch, wenn wir $k(\sqrt[p]{s_a}, \sqrt[p]{s_i}, i\in I_a)$ durch $k(\sqrt[p]{s_0}, \sqrt[p]{s_i}, i\in I_a)$ ersetzen.

F\"{u}r $i\in \{0,\ldots,m\}$ sei $\T_i \subset G(k_{S_0}^{T,\el}|k)$ die Tr\"{a}gheitsgruppe von $\p_i$. Da die Elemente $\Frob_{\p_j}$, $j\neq i$, immer noch $G(k_T^\el|k)$ erzeugen, gilt $V_{{S_0}\backslash \{\p_i\}}^T(k)=0$  nach Lemma~\ref{einheitswda}. Daher verzweigt $\p_i$ in $k_{S_0}^{T,\el}|k$ und $\T_i$ ist zyklisch von der Ordnung~$p$. Nach Konstruktion erzeugen die $m-n+1$ vielen zyklischen Untergruppen der Ordnung~$p$
\[
\T_i, \ i \notin \{ a\} \cup I_a\,,
\]
den $(m-n+1)$-di\-men\-sio\-na\-len Vektorraum $G(k_{S_0}^{T,\el}|k)$ und die gleiche Aussage ist auch f\"{u}r die Untergruppen
\[
\T_i, \ i \notin \{ 0\} \cup I_a\,,
\]
richtig.
Daher erzeugen die Untergruppen $\T_i$,  $i \notin \{ 0, a\} \cup I_a$ einen $(m-n)$-di\-men\-sio\-na\-len Unterraum, die Erweiterung $k_{\{\p_0,\p_a\}}^{T,\el}$ ist zyklisch von der Ordnung $p$ und bei $\p_0$ und $\p_a$ verzweigt.
Wir betrachten nun f\"{u}r $a=1,\ldots, m$ und ein Primideal $\q$ die folgende Bedingung $(C_a)$:

\medskip
\begin{compactitem}
\item $\q \notin T\cup \M$. \smallskip
\item $\q$ zerf\"{a}llt vollst\"{a}ndig in $k_T^\el|k$.\smallskip
\item F\"{u}r jedes $i\in I_a$ zerf\"{a}llt $\q$  in $k(\sqrt[p]{s_i})|k$.\smallskip
\item $\q$ ist tr\"{a}ge in der Erweiterung $k(\sqrt[p]{s_a})|k$.\smallskip
\item Das Bild von $\Frob_{\q}$ in $G(k_{S_0}^{T,\el}|k)$  liegt in $G(k_{S_0}^{T,\el}|k_{\{\p_0,\p_a\}}^{T,\el})\sm \{ 0\}$.\smallskip
\end{compactitem}

\medskip\noindent
Mit Hilfe des Tschebotarjowschen Dichtigkeitssatzes f\"{u}r die elementar-abelsche Erweiterung $k_T^\el(\sqrt[p]{s_0},\ldots,\sqrt[p]{s_m} )|k$  finden wir ein Primideal $\q_1$ in~$k$, das Bedingung $(C_1)$ erf\"{u}llt.  Nach Lemma~\ref{dual} ist die Erweiterung $k_{\{\q_1\}}^{T,\el}$ zyklisch von der Ordnung $p$ und bei $\q_1$ verzweigt. Nun w\"{a}hlen wir nacheinander Stellen $\q_2,\ldots,\q_m$ in $k$, so dass gilt:

\medskip
\begin{compactitem}
\item $\q_a$ erf\"{u}llt Bedingung $(C_a)$, und \smallskip
\item $\q_a$ ist f\"{u}r $b<a$  zerlegt in $k_{\{\q_b\}}^{T,\el}|k$.
\end{compactitem}

\medskip\noindent
Insbesondere sind die $\q_i$ paarweise verschieden und es gilt $N(\q_i)\equiv 1 \bmod p$. Wir behaupten, dass
\[
S=\{\p_0,\ldots,\p_m,\q_1,\ldots,\q_m\}
\]
die gew\"{u}nschten Eigenschaften hat. Es gilt $h^2(X\sm S,T)=2m$ und nach Korollar~\ref{h2lokali} ist die nat\"{u}rliche Abbildung
\[
H_\et^2(X\sm S, T) \lang \prod_{i=1}^m H^2(k_{\p_i}) \oplus \prod_{i=1}^m H^2(k_{\q_i} )
\]
ein Isomorphismus.

\noindent
Es sei $\eta_a$ ein Erzeuger von $H_\et^1(X\sm \{ \q_a\}, T)$. Wir betrachten den durch $\eta_1,\ldots,\eta_m$ in $H_\et^1(X\sm S, T)$ aufgespannten $m$-di\-men\-sio\-na\-len Vektorraum $V$. Dann gilt
\[
H_\et^1(X\sm S, T) \cong H_\et^1(X\sm S_0, T) \oplus V.
\]
Nach Lemma~\ref{dual} gilt
\[
k_T^\el k_{\{\q_a\}}^{T,\el} = k_T^\el(\sqrt[p]{s_{\q_a}}).
\]
Satz~\ref{komponenten} impliziert daher f\"{u}r $a,b, i\in \{1,\ldots, m\} $ das Verschwinden
\[
(\eta_a \cup \eta_b)_{\q_i}= 0 = (\eta_a \cup \eta_b)_{\p_i}.
\]
Es bleibt zu zeigen, dass das Cup-Produkt
\[
H_\et^1(X\sm S_0, T) \otimes V \lang H_\et^2(X\sm S, T)
\]
surjektiv ist. Dazu w\"{a}hlen wir f\"{u}r jedes $a\in \{1,\ldots, m\}$ einen Erzeuger $\chi_a$ von $H_\et^1(X\sm \{\p_0,\p_a\}, T)$. Nach Konstruktion gilt $\chi_a(\T_a)\neq 0$, $\chi_a(\T_i)=0$ f\"{u}r $i\notin \{0,a\}\cup I_a$ und $\chi_a(\Frob_{\q_a})=0$. Desweiteren w\"{a}hlen wir $\psi_a\in H_\et^1(X\sm S, T)$ mit $\psi_a(\Frob_{\q_a})\neq 0$.

Wir behaupten, dass die Elemente $\chi_1\cup \eta_1,\ldots, \chi_m\cup \eta_m, \psi_1\cup \eta_1,\ldots, \psi_m\cup \eta_m$ den $2m$-dimensionalen Vektorraum $H_\et^2(X\sm S, T) $ erzeugen. Dies sieht man in genau der gleichen Weise wie im Fall $\delta=0$ und auch der Rest des Beweises ist von diesem Punkt an w\"{o}rtlich der gleiche.
\end{proof}

\section{Beweis von Theorem~\ref{haupt}}

In diesem Abschnitt beweisen wir Theorem~\ref{haupt}. Wir f\"{u}hren die folgende Bezeichnung ein: es sei  $K|k$ eine (typischerweise unendliche) separable algebraische Erweiterung und $S$ eine endliche Stellenmenge von $k$.
Dann schreiben wir (die Koeffizienten $\F_p$ werden wie vorher ausgelassen)
\[
\ressum_{{\mathfrak p} \in S(K)} H^i(K_\p)
\stackrel{df}{=}
\varinjlim_{k'\subset K} \bigoplus_{{\mathfrak p} \in S(k')} H^i(k'_\p),
\]
wobei sich der Limes auf der rechten Seite \"{u}ber alle endlichen Teilerweiterungen $k'|k$ von $K|k$ erstreckt.
Ist $K|k$ Galoissch mit Gruppe $G(K|k)$, so ist die verallgemeinerte Summe  $\ressumsmall_{{\mathfrak p} \in S(K)} H^i(K_\p)$ gleich dem maximalen diskreten $\Gal(K|k)$-Untermodul des Produkts  $\prod_{{\mathfrak p} \in S(K)} H^i(K_\p)$.

\medskip
Der n\"{a}chste Satz f\"{u}hrt  Theorem~\ref{haupt} auf Theorem~\ref{hilf} zur\"{u}ck.

\begin{proposition}\label{erweiterung}
Es seien $S$, $T$, $S_0$ und $T_0$ paarweise disjunkte endliche Stellenmengen des globalen K\"{o}rpers~$k$. Angenommen $(X\sm S_0, S\cup T \cup T_0)$ hat die $K(\pi,1)$-Eigenschaft f\"{u}r $p$. Dann haben auch $(X\sm S_0, S\cup T)$, $(X\sm S_0, T)$ und $(X\sm (S\cup S_0), T)$ die $K(\pi,1)$-Eigenschaft  f\"{u}r $p$. Es realisiert $k_{S\cup S_0}^T(p)$ die maximale $p$-Erweiterung $k_\p(p)$ f\"{u}r alle\/ $\p\in S$ und der nat\"{u}rliche Homomorphismus
\[
\freeproductmed_{\p\in S(k_{S_0}^{S\cup T}(p))}G(k_\p(p)|k_\p) \lang G\big(k_{S\cup S_0}^T(p)|k_{S_0}^{S\cup T}(p)\big).
\]
ist ein Isomorphismus.
\end{proposition}

\begin{proof} Zwecks Vereinfachung der Bezeichnungen schreiben wir $k_S^T$ f\"{u}r $k_S^T(p)$. Wir berechnen zun\"{a}chst die Kohomologie von $(X\sm S_0, S\cup T)_{k_{S_0}^{S\cup T\cup T_0}}$.
Da die Kurve $(X\sm S_0, S\cup T \cup T_0)$ die $K(\pi,1)$-Eigenschaft f\"{u}r $p$ hat, folgt aus Lemma~\ref{kpi1lem}
\[
H^i_\et\big((X\sm S_0, S\cup T \cup T_0)_{k_{S_0}^{S\cup T \cup T_0}}\big)=0,\quad i\geq 1.
\]
Korollar~\ref{ausschneid2} zeigt daher $H^i_\et((X\sm S_0,S\cup T)_{k_{S_0}^{S\cup T \cup T_0}})=0$ f\"{u}r $i\geq 2$ und einen Isomorphismus
\[
H^1_\et\big((X\sm S_0,S\cup T)_{k_{S_0}^{S\cup T \cup T_0}}\big)\liso
 \ressum_{\p\in T_0(k_{S_0}^{S\cup T \cup T_0})} \back
H^1_\nr(k_\p).
\]
Insbesondere realisiert $k_{S_0}^{S\cup T}$  die maximale elementar-abelsche unverzweigte $p$-Er\-wei\-terung $k_\p^{\nrel}$ von $k_\p$ f\"{u}r alle $\p\in T_0$. Die Hochschild-Serre-Spektralfolge zeigt die Inklusion
\[
H^2\big(G(k_{S_0}^{S\cup T}|k_{S_0}^{S\cup T\cup T_0})\big)\hookrightarrow H^2_\et\big((X\sm S_0,S\cup T)_{k_{S_0}^{S\cup T \cup T_0}}\big)=0,
\]
weshalb die Pro-$p$-Gruppe $G(k_{S_0}^{S\cup T}|k_{S_0}^{S\cup T\cup T_0})$  frei ist. Nach Lemma~\ref{kpi1crit} folgt, dass $(X\sm S_0,S\cup T)_{k_{S_0}^{S\cup T \cup T_0}}$ die $K(\pi,1)$-Eigenschaft f\"{u}r $p$ hat. Daher erhalten wir
\[
H^i_\et\big((X\sm S_0,S\cup T)_{k_{S_0}^{S\cup T}}\big)=0,\quad i\geq 1,
\]
was nach Lemma~\ref{kpi1lem} die $K(\pi,1)$-Eigenschaft f\"{u}r $(X\sm S_0, S\cup T)$ zeigt. Die Zerlegungsgruppen $Z_\p(k_{S_0}^{S\cup T}|k_{S_0}^{S\cup T\cup T_0})$ der (unverzweigten) Stellen $\p\in T_0$ sind nichttrivial und torsionsfrei, weshalb der K\"{o}rper $k_{S_0}^{S\cup T}$ die maximale unverzweigte $p$-Erweiterung $k_\p^\nr(p)$ von $k_\p$ f\"{u}r alle $\p\in T_0$ realisiert.

Die gleichen Argumente zeigen nun auch, dass $(X\sm S_0, T)$ die $K(\pi,1)$-Ei\-gen\-schaft f\"{u}r $p$ hat und dass $k_{S_0}^T$ die maximale unverzweigte $p$-Erweiterung $k_\p^\nr(p)$ von $k_\p$ f\"{u}r alle $\p\in S$ realisiert. Mit Hilfe der Ausschneidungsfolge erhalten wir Isomorphismen
\[
H^i_\et\big((X\sm (S\cup S_0),T)_{k_{S_0}^{T}}\big)\liso
\ressum_{\p\in S(k_{S_0}^{ T})}  H^i_{/\nr}(k_\p),\quad i\geq 1.
\]
Nun realisiert $k_{S_0}^{T}$ die maximalen unverzweigten $p$-Erweiterungen $k_\p^\nr(p)$ der Stellen $\p\in S$, weshalb diese Kohomomologiegruppen f\"{u}r $i\geq 2$ verschwinden, und $k_{S\cup S_0}^T$ realisiert die maximale elementar-abelsche $p$-Erweiterung von $k_\p^\nr(p)$ f\"{u}r alle Stellen $\p\in S$. Analog wie oben erhalten wir, dass $(X\sm(S\cup S_0), T)$ die $K(\pi,1)$-Eigenschaft f\"{u}r $p$ hat und dass die Pro-$p$-Gruppe $G(k_{S\cup S_0}^T|k_{S_0}^T)$ frei ist. Die nat\"{u}rlichen Homomorphismen
\[
G\big(k_\p(p)|k_\p^\nr(p)\big) \longrightarrow Z_\p(k_{S\cup S_0}^T|k_{S_0}^T),\quad \p\in S(k_{S_0}^T),
\]
von den vollen lokalen Gruppen auf die Zerlegungsgruppen sind daher Homomorphismen zwischen freien Pro-$p$-Gruppen, die Isomorphismen auf $H^1(-,\F_p)$ induzieren. Folglich sind sie Isomorphismen (siehe \cite{NSW}, 1.6.15). Wir schlie{\ss}en, dass $k_{S\cup S_0}^T$ f\"{u}r jedes $\p\in S$ die maximale $p$-Erweiterung $k_\p(p)$ realisiert. Nun betrachten wir eine weitere Ausschneidungsfolge, um  Isomorphismen
\[
H^i_\et\big((X\sm (S\cup S_0),T)_{k_{S_0}^{S\cup T}}\big)\liso \ressum_{\p\in S(k_{S_0}^{S\cup T})} \back H^i(k_\p)
\]
zu erhalten. Da $(X\sm (S\cup S_0), T)$ die $K(\pi,1)$-Eigenschaft f\"{u}r $p$ hat, stimmt die Kohomologie der Gruppe $G(k_{S\cup S_0}^T|k_{S_0}^{S\cup T})$ mit der \'{e}talen Kohomologie der Pro-Kurve $(X\sm (S\cup S_0),T)_{k_{S_0}^{S\cup T}}$ \"{u}berein. Unter Ausnutzung der Berechnung der Kohomologie eines freien Produktes (\cite{NSW}, Theorem 4.3.14) schlie{\ss}en wir, dass
\[
\phi: \freeproductmed_{\p\in S(k_{S_0}^{S\cup T})}G(k_\p(p)|k_\p) \lang G(k_{S\cup S_0}^T|k_{S_0}^{S\cup T})
\]
ein Homomorphismus zwischen Pro-$p$-Gruppen ist, der Isomorphismen auf $H^i(-,\F_p)$ f\"{u}r alle $i$ induziert.  Nach \cite{NSW}, Proposition 1.6.15, ist $\phi$ ein Isomorphismus.
\end{proof}

Nun folgern wir Theorem~\ref{haupt}. Es seien $S$, $T$ und $\M$ paarweise disjunkte Stellenmengen, wobei $S$ und $T$ endlich seien und $\M$ die Dirichletdichte $\delta(\M)=0$ habe.

\noindent
Wir w\"{a}hlen uns zun\"{a}chst Stellenmengen $S_0$, $T_0$ zu $S\cup T$ und $\M$ wie in Theorem~\ref{hilf}, d.h.\

\smallskip
\begin{compactitem}
\item $S_0$ ist eine nichtleere Menge von Stellen $\p$  der Norm $N(\p)\equiv 1\bmod p$, \smallskip
\item $S_0\cap (S\cup T\cup T_0\cup \M)=\varnothing$, \smallskip
\item $(X\sm S_0, S\cup T\cup T_0)$ hat die $K(\pi,1)$-Eigenschaft f\"{u}r $p$, \smallskip
\item jedes $\p\in S_0$ verzweigt in  $k_{S_0}^{S\cup T\cup T_0}(p)$, \smallskip
\item $V_{S_0}^{S\cup T\cup T_0}(k)=0$,\smallskip
\item das Cup-Produkt
\[
H^1\big(G_{S_0}^{S\cup T\cup T_0}(k)(p)\big) \otimes H^1\big(G_{S_0}^{S\cup T\cup T_0}(k)(p)\big) \lang H^2\big(G_{S_0}^{S\cup T\cup T_0}(k)(p)\big)
\]
ist surjektiv.
\end{compactitem}

\smallskip\noindent
Dann liefert uns Satz~\ref{erweiterung} die folgenden Aussagen aus Theorem~\ref{haupt}:
wir erhalten $cd\, G_{S\cup S_0}^T(p) \leq 2$, Aussage (ii) aber zun\"{a}chst nur f\"{u}r Stellen aus $S$, sowie die Aussagen (iii) und (iv).

Wir \"{u}berzeugen uns nun, dass $k_{S\cup S_0}^T(p)$ f\"{u}r alle Stellen $\p\in S_0$ die maximale $p$-Erweiterung $k_\p(p)$ von $k_\p$ realisiert. Es sei $\p\in S_0$. Dann liegt $\p$ nicht \"{u}ber~$p$ und der lokale K\"{o}rper $k_\p$ enth\"{a}lt eine primitive $p$-te Einheitswurzel. Die Zerlegungsgruppe $Z_\p(k_{S\cup S_0}^T(p)|k)$ hat als Untergruppe von $G(k_{S\cup S_0}^T(p)|k)$  kohomologische Dimension kleiner gleich~$2$.  Geht man in \cite{NSW}, Theorem~7.5.2, zur maximalen Pro-$p$-Faktorgruppe \"{u}ber, so sieht man, dass die volle lokale Gruppe $\Gal(k_\p(p)|k_\p)$ als Pro-$p$-Gruppe durch zwei Erzeuger $\sigma, \tau$ mit der einen Relation $\sigma\tau\sigma^{-1}=\tau^q$ dargestellt werden kann. Das Element $\tau$ ist ein Erzeuger der Tr\"{a}gheitsgruppe,  $\sigma$ ist eine Hebung des Frobeniusautomorphismus und  $q=N(\p)$. Daher hat $\Gal(k_\p(p)|k_\p)$ genau drei Faktorgruppen von kohomologischer Dimension kleiner gleich $2$:  sich selbst, die triviale Gruppe und die Galoisgruppe der maximal unverzweigten $p$-Erweiterung von~$k_\p$. Da nun $\p$ in der  Erweiterung $k_{S\cup S_0}^T(p)$ verzweigt, ist die Zerlegungsgruppe voll, d.h.\  $k_{S\cup S_0}^T(p)$ realisiert die maximale $p$-Erweiterung $k_\p(p)|k_\p$. Schlie{\ss}lich ist $S_0$ nicht leer und f\"{u}r $\p\in S_0$ hat die Zerlegungsgruppe $Z_\p(k_{S\cup S_0}^T(p)|k)$ kohomologische Dimension~$2$, weshalb dies auch f\"{u}r $G(k_{S\cup S_0}^T(p)|k)$ gilt.
Mit $V_{S_0}^{S\cup T\cup T_0}(k)$ verschwinden auch die Gruppen $V_{S_0}^{S\cup T}(k)$ und $V_{S\cup S_0}^{T}(k)$. Daher sind die Folgen
\[
0\to H_\et^2(X\sm S_0, S\cup T) \to H_\et^2(X\sm (S\cup S_0),  T)\to \bigoplus_{\p\in S} H_\et^2(k_\p)\to 0,
\]
\[
0\to H_\et^1(X\sm S_0, S\cup T) \to H_\et^1(X\sm (S\cup S_0),  T)\to \bigoplus_{\p\in S} H_\et^1(k_\p)\to 0
\]
exakt. Desweiteren ist das Cup-Produkt
\[
H_\et^1(X\sm S_0, S\cup T) \otimes H_\et^1(X\sm S_0, S\cup T) \lang H_\et^2(X\sm S_0, S\cup T)
\]
surjektiv. Dies folgt aus der Wahl von $S_0$ und da der Homomorphismus
\[
H_\et^2(X\sm S_0, S\cup T\cup T_0) \lang H_\et^2(X\sm S_0, S\cup T).
\]
surjektiv ist. Schlie{\ss}lich sind die lokalen Cup-Pro\-duk\-te $H_\et^1(k_\p)\otimes H_\et^1(k_\p)\to H_\et^2(k_\p)$ stets surjektiv und die Inflationshomomorphismen $H^i(G(k_\p(p)|k_\p))\to H^i(k_\p)$ Isomorphismen f\"{u}r alle $i$ (siehe \cite{NSW}, 7.5.8). Da $k_{S\cup S_0}^T(p)$ bei den Stellen in $S$ die maximale lokale $p$-Erweiterung realisiert, erh\"{a}lt man hieraus die Surjektivit\"{a}t des Cup-Produkts
\[
H_\et^1(X\sm (S\cup S_0),  T) \otimes H_\et^1(X\sm (S\cup S_0), T) \lang H_\et^2(X\sm (S\cup S_0), T)
\]
Dies beendet den Beweis von Theorem~\ref{haupt}.

\section{Erweiterung der Stellenmenge} \label{erwsec}

\begin{proposition} \label{enlarge} Es seien $T$ und $S'$ disjunkte endliche Stellenmengen des globalen K\"{o}rpers~$k$, $S\subset S'$  eine Teilmenge und $p\neq \text{char}(k)$ eine Primzahl. Im Zahlk\"{o}rperfall sei $p\neq 2$ oder $k$ total imagin\"{a}r. Die markierte arithmetische Kurve $(X\sm S,T)$ habe die $K(\pi,1)$-Eigenschaft f\"{u}r $p$. Zerf\"{a}llt keine Stelle  $\p\in S'\sm S$ vollst\"{a}ndig in der Erweiterung $k_S^T(p)|k$, so gilt folgendes:

\medskip
\begin{compactitem}
\item[{\rm (i)}] Auch $(X\sm S',T)$ hat die $K(\pi,1)$-Eigenschaft f\"{u}r $p$.\smallskip
\item[{\rm (ii)}] $k_{S'}^T(p)_\p=k_\p(p)$ f\"{u}r alle $\p \in S'\sm S$.
\end{compactitem}

\medskip\noindent
Desweiteren gilt die arithmetische Form des Riemannschen Existenzsatzes, d.h.\ f\"{u}r  $K=k_S^T(p)$ ist der nat\"{u}rliche Homomorphismus
\[
\freeproductmed_{{\mathfrak p} \in S'\backslash S(K)} \Gal(K_\p(p)|K_\p) \longrightarrow \Gal(k_{S'}^T(p)|K)
\]
ein Isomorphismus. Insbesondere ist $\Gal(k_{S'}^T(p)|k_S^T(p))$ eine freie Pro-$p$-Gruppe.
\end{proposition}

\begin{remark}
Ist $k_S^T(p)|k$ unendlich, so hat die Menge der voll zerfallenden Stellen die Dirichletdichte Null. Diese Aussage kann noch im Stil von \cite{Ih} versch\"{a}rft werden, siehe \cite{TV}, Proposition~3.1. Es stellt sich die Frage, ob diese Menge endlich oder sogar gleich $T$ ist. Dies ist unmittelbar klar im Fall $S\supset S_p$, $T=\varnothing$, weil dann die zyklotomische $\Z_p$-Erweiterung von $k$ in $k_S^T(p)$ enthalten ist.
\end{remark}

\begin{proof}[Beweis von Satz~\ref{enlarge}] Die $K(\pi,1)$-Eigenschaft impliziert
\[
H^i(G_S^T(k)(p),\F_p) \cong H^i_\et(X\sm S,T,\F_p)=0 \text{ f\"{u}r } i\geq 4,
\]
insbesondere gilt $\cd\, G_S^T(k)(p)\leq 3$. Sei $\p \in S'\sm S$ und $K=k_S^T(p)$. Nach Annahme zerlegt sich $\p$ nicht vollst\"{a}ndig in $K|k$.   Wegen  $\cd\, G_S^T(k)(p)< \infty$ ist die Zerlegungsgruppe von $\p$ in $K|k$ eine nichttriviale und torsionsfreie Faktorgruppe von  $\Z_p\cong \Gal(k_\p^{\nr}(p)|k_\p)$. Daher gilt
\[
K_\p=k_\p^{\nr}(p)
\]
f\"{u}r jedes $\p\in S'\sm S$. Wir betrachten die Ausschneidungsfolge f\"{u}r  $(X \sm S,T)_{K}$ und $(X \sm S',T)_{K}$. Da  $(X\sm S,T)$ die $K(\pi,1)$-Eigenschaft f\"{u}r  $p$ hat, erhalten wir $H^i_{\et}((X \sm S,T)_{K},\F_p)=0$ f\"{u}r $i \geq 1$.  Unter Auslassung der Koeffizienten $\F_p$ erhalten wir Isomorphismen
\[
H^i_{\et}\big((X \sm S',T)_K\big) \stackrel{\sim}{\longrightarrow} \ressum_{{\mathfrak p} \in S'\backslash S (K)} H^{i+1}_{\mathfrak p}\big((X \sm S,T)_{K}\big)
\]
f\"{u}r $i\geq 1$. Hieraus folgt
\[
H^i_{\et}\big((X\sm S',T)_{K}\big)=0\
\]
f\"{u}r $i\geq 2$. Nun ist $ (X\sm S',T)_{k_{S'}^T(p)}$ die universelle Pro-$p$-\"{U}berlagerung der markierten Kurve
$(X\sm S',T)_{K}$. Daher liefert die  Hochschild-Serre-Spektralfolge
\[
E_2^{ij}=H^i\big(\Gal(k_{S'}^T(p)|K), H^j_\et((X\sm S',T)_{k_{S'}^T(p)})\big) \Rightarrow H^{i+j}_\et((X\sm S',T)_{K})
\]
eine Inklusion
\[
H^2\big(\Gal(k_{S'}^T(p)|K)\big) \hookrightarrow H^2_{\et}\big((X\sm S',T)_{K}\big)=0.
\]
Folglich ist  $\Gal(k_{S'}^T(p)|K)$ eine freie Pro-$p$-Gruppe und wir haben einen Isomorphismus
\[
H^1(\Gal(k_{S'}^T(p)|K)) \stackrel{\sim}{\to} H^1_{\et}\big((X\sm S',T)_{K}\big)
\cong \ressum_{{\mathfrak p}\in S'\backslash S(K)} H^1(K_{\mathfrak p}).
\]
Wir betrachten nun den nat\"{u}rlichen Homomorphismus
\[
\phi: \freeproductmed_{{\mathfrak p} \in S'\backslash S(K)} \Gal(K_\p(p)|K_\p) \longrightarrow \Gal(k_{S'}^T(p)|K).
\]
Wegen $K_\p=k_\p^\nr(p)$ f\"{u}r $\p\in S'\sm S$ sind die Faktoren im freien Produkt auf der linken Seite freie Pro-$p$-Gruppen. Nach der Berechnung der Kohomologie eines freien Produktes (\cite{NSW}, 4.3.10 und 4.1.4) ist $\phi$ ein Homomorphismus zwischen freien Pro-$p$-Gruppen, der einen Isomorphismus auf $H^1(-,\F_p)$ induziert.  Daher ist $\phi$ ein Isomorphismus (siehe \cite{NSW}, 1.6.15).  Insbesondere gilt $k_{S'}^T(p)_\p=k_\p(p)$ f\"{u}r jedes $\p\in S'\sm S$.
Benutzt man nun die Freiheit von $\Gal(k_{S'}^T(p)|K)$, so liefert die Hochschild-Serre-Spektralfolge
einen Isomorphismus
\[
0=H^2_\et((X\sm S',T)_{K})\mapr{\sim} H^2_\et\big((X\sm S',T)_{k_{S'}(p)}\big)^{\Gal(k_{S'}^T(p)|K)}.
\]
Da $\Gal(k_{S'}(p)|k_S(p))$ eine Pro-$p$-Gruppe ist, folgt $H^2_\et\big((X\sm S',T)_{k_{S'}(p)}\big)=0$. Nach Lemma~\ref{kpi1lem} hat somit  $(X\sm S',T)$ die $K(\pi,1)$-Eigenschaft f\"{u}r $p$.
\end{proof}

\section{Dualit\"{a}t f\"{u}r die Fundamentalgruppe}\label{dualsec}
Zun\"{a}chst untersuchen wir den Zusammenhang zwischen der $K(\pi,1)$-Eigenschaft und den universellen Normen globaler Einheiten.

Wir beginnen damit, redundante Stellen aus $S$ zu entfernen: Ist $\p \nmid p$ eine Stelle mit $\zeta_p\notin k_\p$, dann ist jede $p$-Erweiterung des lokalen K\"{o}rpers~$k_\p$ unverzweigt (siehe \cite{NSW}, 7.5.9).  Daher k\"{o}nnen Stellen  $\p \notin  S_p$ mit $N(\p) \not \equiv 1 \bmod p$  in einer $p$-Erweiterung nicht verzweigen.
Durch Entfernung aller dieser redundanten Stellen aus $S$ erhalten wir eine Teilmenge  $S_{\min} \subset S$ mit $G_S^T(p)=G_{S_{\min}}^T(p)$. Im Fall $\delta=1$ gilt $S_\min=S$, d.h.\ es gibt keine redundanten Stellen.

\begin{lemma}
Es hat $(X\sm S,T)$ genau dann die $K(\pi,1)$-Eigenschaft f\"{u}r $p$, wenn dies f\"{u}r $(X\sm S_\min,T)$ der Fall ist.
\end{lemma}

\begin{proof}
Es gilt $k_S^T(p)=k_{S_\min}^T(p)$. Bezeichnen wir diesen K\"{o}rper mit $K$, so gilt nach Satz~\ref{localcoh}
\[
H^i_\p\big((X,T)_K,\F_p\big)=0
\]
f\"{u}r $i\geq 1$ und jedes $\p \in S\sm S_\min(K)$. Die Ausschneidungsfolge zeigt nun, dass f\"{u}r $i\geq 1$ die Gruppe
$H^i_\et((X\sm S,T)_K,\F_p)$ genau dann verschwindet, wenn  dies f\"{u}r $H^i_\et((X\sm S_\min,T)_K,\F_p)$ der Fall ist. Die Aussage folgt daher aus Lemma~\ref{kpi1lem}.
\end{proof}

\begin{proposition} \label{thmb} Es seien $S$ und $T$ disjunkte endliche Stellenmengen des globalen K\"{o}rpers $k$ und $p\neq \text{char}(k)$ eine Primzahl. Im Zahlk\"{o}rperfall sei $p\neq 2$ oder $k$ total imagin\"{a}r.  Dann implizieren je zwei der folgenden Bedingungen  {\rm (a)--(c)} die jeweils dritte.

\medskip\noindent
\begin{compactitem}
\item[\rm (a)] $(X\sm S,T)$ hat die $K(\pi,1)$-Eigenschaft f\"{u}r $p$. \smallskip
\item[\rm (b)] $\varprojlim_{K\subset k_S^T(p)} E_{K,T} \otimes \Z_p=0$. \smallskip
\item[\rm (c)] $k_{S}^T(p)_\p=k_\p(p)$ f\"{u}r alle $\p\in S_\min$.
\end{compactitem}

\medskip\noindent
In {\rm (b)} erstreckt sich der Limes \"{u}ber alle endlichen Teilerweiterungen $K|k$ von $k_S^T(p)|k$.
Gelten {\rm (a)--(c)}, so gilt auch
\[
\varprojlim_{K\subset k_S^T(p)} E_{K,S_\min\cup T} \otimes \Z_p=0.
\]
\end{proposition}

\begin{remarks} 1. Theorem~\ref{haupt} besagt, dass Bedingungen (a)--(c) nach Hinzunahme endlich vieler Stellen zu $S$ gelten.

\noindent
2. Es gelte $\zeta_p \in k$, $S_p\subset S$ und $T=\varnothing$. Dann gilt (a) und Bedingung (c) gilt f\"{u}r  $p>2$, falls $\# S > r_2+2$ (siehe \cite{NSW} 10.9.1 und Remark~2 nach 10.9.3). Im Fall  $k=\Q(\zeta_p)$, $S=S_p$, $T=\varnothing$, gilt Bedingung (c) genau dann, wenn  $p$ eine irregul\"{a}re Primzahl ist.
\end{remarks}

\begin{proof}[Beweis von Satz~\ref{thmb}] Wir k\"{o}nnen ohne Einschr\"{a}nkung  $S=S_\min$ annehmen. Wendet man das  topologische Nakayama-Lemma (\cite{NSW}, 5.2.18) auf den kompakten $\Z_p$-Modul $\varprojlim E_{K,T} \otimes \Z_p$ an, so sieht man, dass Bedingung (b) zur nachfolgenden Bedingung (b') \"{a}quivalent ist:

\medskip
\begin{compactitem}
\item[\rm (b)']\quad  $\varprojlim_{K\subset k_S^T(p)} E_{K,T}/p=0$.
\end{compactitem}

\medskip\noindent
Desweiteren ist nach Lemma~\ref{kpi1lem} Bedingung (a) \"{a}quivalent zu

\medskip
\begin{compactitem}
\item[\rm (a)'] \quad $\varinjlim_{K\subset k_S^T(p)} H^i_\et((X\sm S, T)_K, \F_p)=0$ f\"{u}r $i\geq 1$.
\end{compactitem}

\medskip\noindent
Nach Satz~\ref{globcoh} gilt (a)' f\"{u}r $i=1$, $i\geq 4$, und auch f\"{u}r $i=3$,  falls  $S$ nichtleer oder $\delta=0$ ist. Im Fall $S=\varnothing$, $\delta=1$, gilt $H^3_\et((X\sm S,T)_K,\F_p)\cong \mu_p^\vee$. Daher ist in diesem Fall Bedingung (a)' f\"{u}r $i=3$ genau dann erf\"{u}llt, wenn die Gruppe $G_S^T(k)(p)$ von unendlicher Ordnung ist.

Nach Lemma~\ref{VSchange} haben wir f\"{u}r jedes $K\subset k_S^T(p)$ die exakte Folge
\[
0\longrightarrow E_{K,T} /p \longrightarrow V_\varnothing^T(K) \longrightarrow \null_p \Cl_T(K) \longrightarrow 0\,.
\]
Der K\"{o}rper $k_S^T(p)$ hat keine unverzweigten $p$-Erweiterungen in denen alle Stellen aus $T$ vollst\"{a}ndig zerfallen. Daher erhalten wir nach Klassenk\"{o}rpertheorie
\[
\varprojlim_{K\subset k_S^T(p)} \null_p\Cl_T(K) \subset \varprojlim_{K\subset k_S^T(p)} \Cl_T(K)\otimes \Z_p=0.
\]
Der Dualit\"{a}tssatz \ref{bstdual} liefert uns daher einen Isomorphismus
\[
\varinjlim_{K\subset k_S^T(p)} H^2_\et\big((X, T)_K, \F_p\big)=\varinjlim_{K\subset k_S^T(p)} \Sha^2(K,\varnothing, T) \cong \big(\varprojlim_{K\subset k_S^T(p)} E_{K,T}/p \ \big)^\vee. \leqno (*)
\]
Wir zeigen zun\"{a}chst, dass im Fall $S=\varnothing$ (in dem (c) trivialerweise gilt) die Bedingungen (a) und (b) \"{a}quivalent sind. Gilt (a)', so folgt mit $(\ast)$ die G\"{u}ltigkeit von (b)'. Gilt (b), so folgt insbesondere, dass $\delta=0$ gilt oder $G_S^T(k)(p)$ unendliche Ordnung hat. Somit erhalten wir (a)' f\"{u}r $i=3$. Au{\ss}erdem folgt (a)' f\"{u}r $i=2$ mit Hilfe von $(\ast)$ aus (b)'. Dies beendet den Beweis im Fall $S=\varnothing$.

\medskip
Von nun an nehmen wir $S\neq \varnothing$ an.
F\"{u}r $\p\in S=S_\min$ besitzt jede echte Galoissche Teilerweiterung von $k_\p(p)|k_\p$ verzweigte $p$-Erweiterungen. Nach der Berechnung der lokalen Kohomologie in Satz~\ref{localcoh} ist daher (c) \"{a}quivalent zu

\medskip
\begin{compactitem}
\item[\rm (c)'] \quad $\varinjlim_{K\subset k_S^T(p)} \bigoplus_{\p \in S(K)} H^i_\p((X,T)_K, \F_p)=0$ f\"{u}r alle $i$,
\end{compactitem}

\medskip\noindent
und auch zu
\medskip
\begin{compactitem}
\item[\rm (c)''] \quad $\varinjlim_{K\subset k_S^T(p)} \bigoplus_{\p \in S(K)} H^2_\p((X,T)_K, \F_p)=0$.
\end{compactitem}

\medskip\noindent
Nun betrachten wir den direkten Limes \"{u}ber $K\subset k_S^T(p)$ der Ausschneidungsfolgen (Koeffizienten $\F_p)$
\[
\cdots \to  \bigoplus_{\p\in S(K)} H^i_\p\big((X,T)_K\big) \to H^i_\et\big((X,T)_K\big) \to H^i_\et\big((X\sm S,T)_K\big) \to \cdots \leqno (**)
\]
Gilt (a)', so verschwinden die rechten Terme in $(**)$ f\"{u}r $i\geq 1$ im Limes und
$(\ast)$ zeigt dann die \"{A}quivalenz zwischen (b)' und (c)''.

Nun m\"{o}gen (b) und (c) gelten. Wie oben impliziert (b) das Verschwinden des mittleren Terms f\"{u}r  $i=2$ im Limes der Folge $(**)$. Bedingung (c)' zeigt dann (a)'. Damit haben wir gezeigt, dass je zwei der Bedingungen (a)--(c) die jeweils dritte implizieren.

\medskip
Schlie{\ss}lich m\"{o}gen (a)--(c) gelten. Tensorieren wir f\"{u}r $K\subset k_S^T(p)$ die exakte Folge (siehe \cite{NSW}, 10.3.12)
\[
0 \to E_{K,T} \to E_{K,S\cup T} \to \bigoplus_{\p\in S(K)} (K_\p^\times/U_\p) \to \Cl_T(K) \to \Cl_{S\cup T}(K) \to 0
\]
mit (der flachen $\Z$-Algebra) $\Z_p$, so erhalten wir eine exakte Folge endlich erzeugter, und daher kompakter, $\Z_p$-Moduln. Geht man nun zum projektiven Limes \"{u}ber alle endlichen Teilerweiterungen $K$ von $k_S^T(p)|k$ \"{u}ber und benutzt $\varprojlim \Cl_T(K) \otimes \Z_p =0$, so erh\"{a}lt man die exakte Folge
\[
0 \to \!\!\! \varprojlim_{K\subset k_S^T(p)}\!\!\! E_{K,T} \otimes \Z_p \to \!\!\!\varprojlim_{K\subset k_S^T(p)}\!\! \! E_{K,S\cup T} \otimes \Z_p \to \!\! \varprojlim_{K\subset k_S^T(p)} \bigoplus_{\p\in S(K)} (K_\p^\times /U_\p) \otimes \Z_p \to 0.
\]
Bedingung (c) und lokale Klassenk\"{o}rpertheorie implizieren das Verschwinden des rechten Limes. Daher impliziert
(b) das Verschwinden des projektiven Limes in der Mitte.
\end{proof}

Gilt $G_S^T(k)(p)\neq 1$ und ist Bedingung (a) aus Satz~\ref{thmb} erf\"{u}llt, dann kann Bedingung (c) nur an Primteilern von $p$ scheitern. Dies folgt aus dem n\"{a}chsten

\begin{proposition} \label{fulllocal} Es seien $S$ und $T$ disjunkte endliche Stellenmengen des globalen K\"{o}rpers $k$ und $p\neq \text{char}(k)$ eine Primzahl. Im Zahlk\"{o}rperfall sei $p\neq 2$ oder $k$ total imagin\"{a}r. Im Funktionenk\"{o}rperfall gelte
\[
\null_p \Cl(k)\neq 0 \ \text{ oder }\ \delta=0\ \text{ oder }\ T\neq\varnothing\ \text{ oder }\ \# S\geq 2.
\]
Hat $(X\sm S,T)$ die $K(\pi,1)$-Eigenschaft f\"{u}r $p$ und gilt $G_S^T(k)(p)\neq 1$, dann hat jede Stelle  $\p\in S$ mit $\zeta_p\in k_\p$ eine unendliche Tr\"{a}gheitsgruppe in $G_S^T(k)(p)$. Desweiteren gilt
\[
k_S^T(p)_\p=k_\p(p)
\]
f\"{u}r jedes  $\p\in S_\min \sm S_p$.
\end{proposition}

\begin{example} Sei $\F$ ein endlicher K\"{o}rper mit $\# \F \equiv 1 \bmod p$.
Wir setzen $k=\F(t)$, $T=\varnothing$ und $S$ bestehe aus der unendlichen Stelle, d.h.\ der Stelle von $k$, die zur Gradbewertung assoziiert ist. Dann hat ${\mathbb A}^1_\F={\mathbb P}^1_{\F} \sm \{\infty\}$ die $K(\pi,1)$-Eigenschaft f\"{u}r~$p$ und $k_{\{\infty \}}(p)|k$ ist die (unverzweigte) zyklotomische $\Z_p$-Erweiterung. Daher ist die in Satz~\ref{fulllocal} im Funktionenk\"{o}rperfall gemachte Voraussetzung n\"{o}tig.
\end{example}

\begin{proof}[Beweis von Satz~\ref{fulllocal}] Ohne Einschr\"{a}nkung k\"{o}nnen wir $S=S_\min\neq \varnothing$ annehmen.
Angenommen ein $\p\in S$ mit $\zeta_p\in k_\p$ w\"{u}rde in der Erweiterung $k_S^T(p)|k$ nicht verzweigen. Dann gilt mit  $S'=S\sm \{\p\}$ die Gleichheit $k_{S'}^T(p)=k_S^T(p)$, insbesondere erhalten wir einen Isomorphismus
\[
H^1_\et (X\sm S', T, \F_p) \stackrel{\sim}{\lang} H^1_\et (X\sm S, T, \F_p).
\]
Im Folgenden schlie{\ss}en wir die Koeffizienten $\F_p$ von der Notation aus.  Unter Verwendung von $H^3_\et(X\sm S,T)=0$ liefert die Ausschneidungsfolge das kommutative und exakte Diagramm
\[
\xymatrix@C=.45cm{&H^2(G_{S'}^T(k)(p))\ar[r]^\sim\ar@{^{(}->}[d]&H^2(G_S^T(k)(p))\ar[d]^\wr\\
H^2_\p(X,T)\ar@{^{(}->}[r]&H^2_\et(X\sm S',T)\ar[r]^\alpha&H^2_\et(X\sm S,T)\ar[r]&H^3_\p(X,T)\ar@{->>}[r]&H^3_\et(X\sm S',T).}
\]
Daher ist $\alpha$ eine spaltende Surjektion und  $\F_p\cong H^3_\p(X,T)\stackrel{\sim}{\to} H^3_\et(X\sm S',T)$. Nach Satz~\ref{globcoh} folgt $S'=\varnothing$, also $S=\{\p\}$, und $\delta=1$. Dieselbe \"{U}berlegung wendet sich auf jede endliche Teilerweiterung $K$ von  $k_S^T(p)|k$ an, weshalb $\p$ unzerlegt in der Erweiterung $k_S^T(p)=k_\varnothing^T(p)$ ist. Daher ist der nat\"{u}rliche Homomorphismus
\[
\Gal(k_\p^\nr(p)|k_\p) \lang G_\varnothing^T(k)(p)
\]
surjektiv, weshalb die Gruppe $G_S^T(k)(p)=G_\varnothing^T(k)(p)$ prozyklisch ist. Im Funktionenk\"{o}rperfall kann daher nicht gleichzeitig  $\null_p \Cl(k)\neq 0$ und $T=\varnothing$ gelten. Wegen $\# S=1$ und $\delta=1$ erhalten wir aus unseren Voraussetzungen, dass $T\neq \varnothing$ im Funktionenk\"{o}rperfall gilt. Daher ist nach Klassenk\"{o}rpertheorie  die Gruppe $G_\varnothing^T(k)(p)$ endlich. Da sie nach Voraussetzung nichttrivial ist, hat sie unendliche kohomologische Dimension, was der $K(\pi,1)$-Eigenschaft widerspricht. Also verzweigt jedes $\p\in S$ mit  $\zeta_p\in k_\p$  in $k_S^T(p)$. Da sich diese \"{U}berlegung auf jede endliche Teilerweiterung von $k$ in $k_S^T(p)$ anwendet, m\"{u}ssen die Tr\"{a}gheitsgruppen unendlich sein.  F\"{u}r $\p\notin  S_p$ impliziert dies  $k_S(p)_\p=k_\p(p)$, wie man leicht an der explizit bekannten Struktur der Gruppe $\Gal(k_\p(p)|k_\p)$ sieht (vgl.\ \cite{NSW}, 7.5.2).
\end{proof}

\begin{proposition} \label{dualmod} Es seien $S\neq \varnothing$ und $T$ disjunkte endliche Stellenmengen des globalen K\"{o}rpers $k$ und $p\neq \text{char}(k)$ eine Primzahl. Im Zahlk\"{o}rperfall sei $p\neq 2$ oder $k$ total imagin\"{a}r. Angenommen es gelten die Bedingungen {\rm (a)--(c)} aus Satz~\ref{thmb} und es gilt $\zeta_p\in k_\p$ f\"{u}r alle $\p\in S$.

\smallskip
Dann ist $G_S^T(k)(p)$ eine Pro-$p$-Dualit\"{a}tsgruppe der Dimension~$2$.
\end{proposition}

\begin{proof} Aus Bedingung (a) folgt  $H^3(G_{S}^T(k)(p),\F_p) \stackrel{\sim}{\to} H^3_\et (X\sm S, T, \F_p)=0$, also $\cd\, G_{S}^T(k)(p)\leq 2$. Andererseits enth\"{a}lt nach (c) die Gruppe $G_{S}^T(k)(p)$ f\"{u}r jedes $\p\in S$ die volle lokale Gruppe $\Gal(k_\p(p)|k_\p)$ als Untergruppe. Wegen $\zeta_p\in k_\p$ f\"{u}r $\p\in S$ haben diese lokalen Gruppen die kohomologische Dimension~$2$, also gilt auch  $\cd\,G_{S}^T(k)(p)$.

Um zu zeigen, dass $G_S^T(k)(p)$ eine Dualit\"{a}tsgruppe ist, m\"{u}ssen wir nach \cite{NSW}, Theorem~3.4.6, das Verschwinden der Terme
\[
D_i\big(G_S^T(k)(p)\big): = \varinjlim_{\substack{U\subset G_S^T(k)(p)\\ \cor^\vee}} H^i(U, \F_p)^\vee,\quad i=0,1,
\]
nachweisen. Hierbei durchl\"{a}uft $U$ die offenen Normalteiler von $G_S^T(k)(p)$ und die \"{U}bergangsabbildungen sind die Duale der Korestriktionshomomorphismen. Das Verschwinden von $D_0$ folgt trivial aus der Unendlichkeit der Gruppe $G_S^T(k)(p)$. Wir haben daher nachzuweisen, dass
\[
\varinjlim_{K\subset k_S^T(p)} H^1\big((X\sm S,T)_K,\F_p\big)^\vee=0
\]
gilt. Wir nehmen zun\"{a}chst an, dass $k$ ein Zahlk\"{o}rper oder $T\neq \varnothing$ ist. In diesem Fall ist f\"{u}r jede endliche Teilerweiterung $K$ von $k_S^T(p)|k$ die Gruppe $\Cl_T(K)$ endlich und nach dem Hauptidealsatz gilt
\[
\varinjlim_{K\subset k_S^T(p)} H^1_\et\big((X,T)_K,\F_p\big)^\vee= \varinjlim_{K\subset k_S^T(p)} \Cl_T(K)/p =0.
\]
Die Ausschneidungsfolge und der lokale Dualit\"{a}tssatz zeigen die Exaktheit der Folge
\[
\bigoplus_{\p\in S(K)} H^1_\nr(K_\p,\mu_p) \to H^1\big((X\sm S,T)_K,\F_p\big)^\vee \to H^1\big((X,T)_K,\F_p\big)^\vee\,.
\]
Nun gilt $\zeta_p\in k_\p$ und  $k_S^T(p)_\p=k_\p(p)$ f\"{u}r alle $\p\in S$, weshalb der linke Term im Limes \"{u}ber alle $K\subset k_S^T(p)$ verschwindet. Das Verschwinden des rechten Terms im Limes haben wir oben eingesehen, was die gew\"{u}nschte Aussage zeigt.

Es verbleibt der Fall $T=\varnothing$, wenn $k$ ein Funktionenk\"{o}rper ist. Nach Poincar\'{e}-Dualit\"{a}t gilt
\[
H^1_\et((X\sm S)_K,\F_p)^\vee \cong H^2_c((X\sm S)_K,\mu_p)= H^2_\et(X_K,j_!\mu_p),
\]
wobei $j: (X\sm S)_K \to X_K$ die Einbettung bezeichnet.
Die Ausschneidungsfolge zusammen mit $H^2_\p(X_K,j_!\mu_p)\cong H^1(K_\p,\mu_p)$ f\"{u}r $\p\in S(K)$ zeigt die Exaktheit der Folge
\[
\bigoplus_{\p \in S(K)} H^1(K_\p,\mu_p) \to H^2_\et(X_K, j_! \mu_p) \to H^2_\et\big((X\sm S)_K, \mu_p\big).\leqno (\ast)
\]
Wieder wegen $\zeta_p\in k_\p$ und $k_S(p)_\p=k_\p(p)$ f\"{u}r $\p\in S$ verschwindet der linke Term von  $(\ast)$ beim \"{U}bergang zum Limes \"{u}ber alle $K\subset k_S(p)$.  Die Kummerfolge induziert die Exaktheit von
\[
\Cl_S(K)/p \lang H^2_\et\big((X\sm S)_K, \mu_p\big) \lang \null_p \Br\big((X\sm S)_K\big).\leqno (\ast\ast)
\]
Der Hauptidealsatz zeigt $\varinjlim_{K\subset k_S(p)} \Cl_S(K)/p=0$ und das Hasseprinzip f\"{u}r die Brauergruppe gibt uns eine Injektion
\[
\null_p \Br\big((X\sm S)_K\big) \hookrightarrow \bigoplus_{\p\in S(K)} \null_p\Br(K_\p).
\]
Da nun $k_S(p)$ f\"{u}r $\p\in S$ die maximale unverzweigte $p$-Erweiterung von $k_\p$ realisiert, verschwindet der rechte, also auch der mittlere Term von $(\ast\ast)$ im Limes \"{u}ber $K$, und folglich auch der mittlere Term von $(\ast)$. Dies beendet den Beweis.
\end{proof}

\begin{remark} F\"{u}r $S\neq \varnothing$ ist die $S$-$T$-Idelklassengruppe von $k_S^T$  definiert durch
\[
C_S^T:=\varinjlim_{K\subset k_S^T} \coker \big(E_{K,S\cup T} \lang \prod_{\p\in S(K)} K_\p^\times\big)\,.
\]
Das Paar $(G_S^T(k), C_S^T)$ ist eine Klassenformation. F\"{u}r einen Teilk\"{o}rper $K\subset k_S^T$ setzt man $C_S^T(K):=(C_S^T)^{\Gal(k_S^T|K)}$.
Unter den Voraussetzungen von Satz~\ref{dualmod} ist der dualisierende Modul der Dualit\"{a}tsgruppe $G_S^T(k)(p)$ isomorph zu
$\text{\rm tor}_{p}\big(C_S^T(k_S^T(p))\big)$.
\end{remark}

\bigskip

\noindent {NWF I - Mathematik, Universit\"{a}t Regensburg, D-93040
Regensburg, Deutschland. Email-Adresse: alexander.schmidt@mathematik.uni-regensburg.de}

\end{document}